\documentclass[10pt,twoside,reqno,centertags,draft]{amsart}
\usepackage{color}
  \usepackage{amsmath,amsthm,amsfonts,amssymb}
\usepackage[notref,notcite]{}
\usepackage{breqn}

\begin{document}
\title{Fractional $p$- Laplacian Kirchhoff-type problem involving a Singular term via Nehari manifold}{}
\date{}
\maketitle
\numberwithin{equation}{section}
\allowdisplaybreaks
\newtheorem{theorem}{Theorem}[section]
\newtheorem{proposition}{Proposition}[section]
 \newtheorem{gindextheorem}{General Index Theorem}[section]
\newtheorem{indextheorem} {Index Theorem}[section]
\newtheorem{standardbasis}{Standard Basis}[section]
\newtheorem{generators} {Generators}[section]
\newtheorem{lemma} {Lemma}[section]
\newtheorem{corollary}{Corollary}[section]
\newtheorem{example}{Example}[section]
\newtheorem{examples}{Examples}[section]
\newtheorem{exercise}{Exercise}[section]
\newtheorem{remark}{Remark}[section]
\newtheorem{remarks}{Remarks}[section]
\newtheorem{definition} {{Definition}}[section]
\newtheorem{definitions}{Definitions}[section]
\newtheorem{notation}{Notation}[section]
\newtheorem{notations}{Notations}[section]
\newtheorem{defnot}{Definitions and Notations}[section]

\def\H{{\mathbb H}}
\def\N{{\mathbb N}}
\def\R{{\mathbb R}}

\newcommand{\be} {\begin{equation}}
\newcommand{\ee} {\end{equation}}
\newcommand{\bea} {\begin{eqnarray}}
\newcommand{\eea} {\end{eqnarray}}
\newcommand{\Bea} {\begin{eqnarray*}}
\newcommand{\Eea} {\end{eqnarray*}}
\newcommand{\p} {\partial}
\newcommand{\ov} {\over}
\newcommand{\al} {\alpha}
\newcommand{\ba} {\beta}
\newcommand{\de} {\delta}
\newcommand{\ga} {\gamma}
\newcommand{\Ga} {\Gamma}
\newcommand{\Om} {\Omega}
\newcommand{\om} {\omega}
\newcommand{\De} {\Delta}
\newcommand{\la} {\lambda}
\newcommand{\si} {\sigma}
\newcommand{\Si} {\Sigma}
\newcommand{\La} {\Lambda}
\newcommand{\no} {\nonumber}
\newcommand{\noi} {\noindent}
\newcommand{\lab} {\label}
\newcommand{\na} {\nabla}
\newcommand{\vp} {\varphi}
\newcommand{\var} {\varepsilon}
\newcommand{\RR}{{\mathbb R}}
\newcommand{\CC}{{\mathbb C}}
\newcommand{\NN}{{\mathbb N}}
\newcommand{\ZZ}{{\mathbb Z}}
\renewcommand{\SS}{{\mathbb S}}
\newcommand{\esssup}{\mathop{\rm {ess\,sup}}\limits}   
\newcommand{\essinf}{\mathop{\rm {ess\,inf}}\limits}   
\newcommand{\weaklim}{\mathop{\rm {weak-lim}}\limits}
\newcommand{\wstarlim}{\mathop{\rm {weak^\ast-lim}}\limits}

\newcommand{\RE}{\Re {\mathfrak e}}   
\newcommand{\IM}{\Im {\mathfrak m}}   
\renewcommand{\colon}{:\,}
\newcommand{\eps}{\varepsilon}
\newcommand{\half}{\textstyle\frac12}
\newcommand{\Takac}{Tak\'a\v{c}}
\newcommand{\eqdef}{\stackrel{{\rm {def}}}{=}}   
\newcommand{\wstarconverge}{\stackrel{*}{\rightharpoonup}}

\newcommand{\Div}{\nabla\cdot}
\newcommand{\Curl}{\nabla\times}
\newcommand{\Meas}{\mathop{\mathrm{meas}}}
\newcommand{\Int}{\mathop{\mathrm{Int}}}
\newcommand{\Clos}{\mathop{\mathrm{Clos}}}
\newcommand{\Lin}{\mathop{\mathrm{lin}}}
\newcommand{\Dist}{\mathop{\mathrm{dist}}}

\newcommand{\Square}{$\sqcap$\hskip -1.5ex $\sqcup$}
\newcommand{\Blacksquare}{\vrule height 1.7ex width 1.7ex depth 0.2ex }
\begin{center}
DJAMEL Abid\\
\bf{E-mail:} djamelabid0312@gmail.com, djamel.abid@univ-setif.dz
\end{center}

\begin{abstract}
this paper is dedicated to studying the existence of nontrivial positive solutions for the following Kirchhoff-type problem with sign change nonlinearities and a singular term.
\begin{equation*}
\begin{cases}
M( \int_{\Omega} \frac{\vert u(x) - u(y) \vert^{p}}{\vert x-y \vert^{n + ps}}dxdy)\left(-\Delta\right)_{p}^{s}u + u^{p-1} = c(x)u^{-\alpha} + \lambda f(x,u) & \textit{in} \quad \Omega\\
u = 0 &  \textit{in} \quad \mathbb{R}^n\Omega.
\end{cases}
\end{equation*}
where  $\Omega \subset \mathbb{R}^n ( n>ps ),$ is a bounded smooth  domain,  $\alpha, s\in (0,1),$ $ \lambda $ is positive parameter $ c :\Omega \longrightarrow \mathbb{R}^+,$ is a nonnegative function with $ c(x) \in L^{\infty},$ and $ f:\Omega\times [0,\infty)\longrightarrow \mathbb{R},$ is  homogenous of order $ q-1,$ snd $ M$ is the Kirchhoff function. Using the Nehari manifold and EkelandS variational principle
we prove that for the appropriate choice of $\lambda$ our problem has at least two positive solutions for both subcritical and critical cases.\newline
\textbf{keyword:} Nehari manifold, fibre maps, Kirchhoff type problem, Singularity, 
\end{abstract}
\section{Introduction}
In this paper, we study the existence and multiplicity of nonnegative solutions for the following Kirchhoff-type problem with a singular term and sign-changing nonlinearity
\begin{equation} \label{p0}
\begin{cases}
M( \int_{\Omega} \frac{\vert u(x) - u(y) \vert^{p}}{\vert x-y \vert^{n + ps}}dxdy)\left(-\Delta\right)_{p}^{s}u + u^{p-1}u = c(x)u^{-\alpha} + \lambda f(x,u) & \textit{in} \quad \Omega\\
u = 0 &  \textit{in} \quad \mathbb{R}^n\Omega.
\end{cases}
\end{equation}
where  $\Omega \subset \mathbb{R}^n,$ is an open bounded smooth domain, $ n > ps,$ with $ s \in (0,1),$ and $ p_s^* = \frac{np}{n-ps},$ is the critical Sobolev exponent, $ \lambda,$ is a positive parameter, the exponent $ \alpha \in (0,1),$ and $ \left(-\Delta\right)_{p}^{s}$ is the fractional $ p$-Lapacian defined as 
\begin{equation*}
 \left(-\Delta\right)_{p}^{s} \phi(x) = \lim_{\epsilon \to 0}\int_{\mathbb{R}^n\setminus B(x,\epsilon)}\frac{\vert \phi(x) -\phi(y)\vert^{p -2}(\phi(x) -\phi(y))}{\vert x -y\vert^{n +ps}}dy,
\end{equation*}
for any $ \phi \in \mathbb{C}_{0}^{\infty}(\mathbb{R}^{n}),$ where  $ B(x,\epsilon)$ is the ball in $ \mathbb{R}^n$ with center $ x$ and radius $ \epsilon.$
for further details on the fractional Laplacian see \cite{di2012hitchhikers}, \cite{xiang2016existence}\newline
Let $ M : (0,+\infty) \longrightarrow (0, +\infty),$ the non-degenerate Kirchhoff function defined as 
\begin{equation}\label{kf}
M(t) = a + b t^{\theta-1},  a > 0 , b> 0 , \textit{for all} \,\,\, t \geq 0 \,\,\, \textit{and} \,\,\, \theta > 1.
\end{equation}
and $ f$ is homogenous function of order $ q-1,$ where $ F(x,u) = \int_{0}^{u} f(x,s)ds,$ is homogenous of order $ q,$ so the Euler identity satisfied, that is 
\begin{equation}\label{Euler}
q F(x,u) = uf(x,u). 
\end{equation}
And by \ref{Euler} we can easily prove that there exist $ \gamma_0 >0$ such that
\begin{equation*}
\vert F(x,u) \vert \leq \gamma_0 \vert u \vert^{q} \leq \gamma \vert u \vert^{q}
\end{equation*}
and without any loss of generality we can take$ \gamma = \gamma_0 +1.$ Then by Holder ineauality we get
\begin{equation}\label{e1}
\int F(x,u) \leq \gamma \Vert u \Vert_{q}^{q} \leq \gamma \vert \Omega \vert^{\frac{p_s^*-q}{p_s^*}}S_{p}^{-\frac{q}{p}} \Vert u \Vert^{q}
\end{equation}
Suppose that for $ t_1, t_2 \in \mathbb{R}^*_+$  we have 
\begin{equation}\label{fd}
\sup_{x \in B_{\delta}(0)}\left\lbrace f(x,t_1)t_2\right\rbrace \leq M 
\end{equation} 
and
\begin{equation}\label{fd2}
0 <\inf_{t >0}\frac{f(x,t)}{t}, \quad \textit{uniformly in} \quad x \in B_{\delta}(0) \subset \Omega,
\end{equation}
for some $ \delta >0,$ also 
\begin{equation}\label{Fd}
F(x, t_1 +t_2) -F(x, t_1) -F(x, t_2) > M,
\end{equation}
and 
\begin{equation}\label{fd1}
\lim_{t \to 0}\frac{F(x,t)}{t^{q}} =\infty, \quad \textit{ uniformly for}  \quad x \in B_{\delta}(0),
\end{equation}
where $ B_{\delta}(0),$ is the ball in $ \mathbb{R}^{n}$ centered at $ 0$ with raduis $ \delta.$ And by Holder inequality we get
\begin{equation}\label{e2}
\int c(x)\vert u\vert^{-\alpha+1} \leq \Vert c \Vert_{L^{\infty}} \Vert u \Vert_{-\alpha +1}^{-\alpha +1} \leq \Vert c \Vert_{L^{\infty}}
\vert \Omega\vert^{\frac{p_s^* +\alpha -1}{p_s^*}}S_{p}^{-\frac{-\alpha+1}{p}} \Vert u \Vert^{-\alpha+1}
\end{equation}
and 
Where $ S_{p}$ is the fractional Sobolev constant defined as
\begin{equation}\label{e3}
S_{p} =\inf_{u \in X_0\setminus \{ 0\}}\frac{\Vert u\Vert^{p}}{\Vert u\Vert^{p}_{p_s^*}}
\end{equation}
Moreover by \cite{mosconi2016brezis},  the function $ u_{\epsilon}$ for $ \epsilon >0$ defined as
\begin{equation}\label{eee1}
u_{\epsilon}(x) =\epsilon^{-\frac{n -ps}{p}}\bar{u}(\frac{\vert x\vert}{\epsilon})\,\, \textit{with} \,\, \bar{u}(x) =\frac{1}{(1 +\vert x\vert^{p^{\prime}})^{\frac{n -ps}{p}}} \,\,\textit{and}\,\, p^{\prime} =\frac{p}{p -1}, \,\, x \in \mathbb{R}^n.
\end{equation}
is a local minimizer for $ S_{p}$ satisfying
\begin{equation}\label{eee2}
\Vert u_{\epsilon}\Vert^{p} =\Vert u_{\epsilon}\Vert_{p_s^*}^{p} =S_{p}^{\frac{n}{sp}}.
\end{equation}
Put
\begin{equation*}
u^+ =\max\{0, u\} \quad \textit{and} \quad u^- =\max\{0, -u\}.
\end{equation*}
\begin{equation*}
X^{-}=\left\lbrace  u \in X_0 : F(x,u) \leq 0 \right\rbrace, \quad X^{+} =\left\lbrace  u \in X_0 : F(x,u) > 0 \right\rbrace,
\end{equation*}
\begin{equation*}
\Omega^{-} =\left\lbrace  x \in \Omega : F(x,u) \leq 0 \right\rbrace, \quad \Omega^{+} =\left\lbrace  x \in \Omega : F(x,u) > 0 \right\rbrace,
\end{equation*}
Recently problems of elliptic equations involving critical exponents or subcritical exponents with the singular term have been widely studied, such as in \cite{saoudi2017multiplicity} the authors studied the existence and multiplicity of positive solutions for the subcritical case to the following problem by the Nehari method
\begin{equation}
\begin{cases}
-div\left( M(x)\nabla u\right) =u^{-\gamma} +\lambda u^{p}, & \quad \textit{in} \quad \Omega,\\
u\vert_{\partial \Omega} = 0, \quad u > 0, & \quad \textit{in}  \quad \Omega, 
\end{cases}
\end{equation} 
for more about the subcritical case see \cite{hirano2004existence}, in particular in \cite{ghanmi2016nehari} the authors by the same approach established the existence of at least two solutions to the following singular problem
\begin{equation}
\begin{cases}
\left(-\Delta\right)^{s} =a(x)u^{-\gamma} +\lambda f(x,u), & \quad \textit{in} \quad \Omega,\\
u = 0, & \quad \textit{in} \quad \mathbb{R}^n\setminus\Omega,
\end{cases}
\end{equation}
In \cite{arora2020polyharmonic} the authors studied a Kirchhoff type Choquard problem involving a critical exponential non-linearity and singular weights using mountain pass theorem, for more details about the subcritical case see also \cite{arcoya2014multiplicity},\cite{godoy2017multiplicity},\cite{choi2019existence},
\cite{figueiredo2012multiplicity}. In \cite{fiscella2019nehari} the authors used the Nehari manifold to obtaining existence and multiplicity for Kirchhoff problem involving singular and critical terms, for further details for problems involving critical and singular terms see \cite{giacomoni2009multiplicity},\cite{lei2015multiple},\cite{mukherjee2016fractional}\cite{giacomoni2017positive}m\cite{fiscella2019fractional}, \cite{autuori2015stationary}, \cite{giacomoni2009multiplicity} and references therein, for singular and supercritical nonlinearity, see \cite{boccardo2012dirichlet},
\cite{assunccao2017existence}.\newline
Inspired by works mentioned above, mostly \cite{fiscella2019nehari}, \cite{ghanmi2016multiplicity} and \cite{saoudi2017multiplicity} we prove the existence and multiplicity solutions for problem \ref{p}, using the Nehari manifold, fibre maps and some variational technique...\newline
we will deal with the following Kirchhoff type problem
\begin{equation}\label{p}
\begin{cases}
\left( a +b\Vert u\Vert^{p(\theta -1)}\right)\left(-\Delta\right)_{p}^{s}u +(u^+)^{p-2}u^+ = c(x)(u^+)^{-\alpha} + \lambda f(x,u^+) & \textit{in} \quad \Omega\\
u = 0 &  \textit{in} \quad \mathbb{R}^n\setminus \Omega.
\end{cases}
\end{equation}
our first result is about the subcritical case
\begin{theorem}\label{th1}
Suppose that $ 1 < p < q < p_s^*,$ $ 0 <\alpha<1,$  and $ f$ satisfying \ref{Euler}, \ref{e1} Then there exist $ \lambda^* > 0$ such that for any $ \lambda \in (0, \lambda^*),$ problem \ref{p0} has at least two nontrivial positive solutions.
\end{theorem}
\begin{theorem}\label{th2}
Suppose that $ 1 < p < q = p_s^*$ and $ s, \alpha \in (0 ,1),$ then 
\begin{itemize}
\item theree exist $ \lambda^{**} >0$ such that problem \ref{p0} has one positive solution for $ \lambda < \lambda^{**},$
\item there exist $ \lambda^{***} >0$ such that problem \ref{p0} has two positive solutions for $ \lambda < \lambda^{***}.$
\end{itemize}
Moreover, in the first case the solution has negative energy, and in the second case our two solutions for \ref{p0} are solutions to \ref{p} for 
appreoprete values of $ b >0.$
\end{theorem}
\section{Preliminaries}
We will find our solutions in the space that takes into consideration the interaction between $ \Omega,$ and its complementary which gives a positive contribution in the 
Gagliardo norm associated to the fractional Sobolev space $ \mathcal{H}^{s} = W^{s,2},$  Hence we define the space 
\begin{equation}
X_0 = \left\lbrace  u \in \mathcal{H}^{s}(\mathbb{R}^n) : u = 0 \,\, \textit{a.e.} \,\, \textit{in} \,\, \mathbb{R}^n \setminus \Omega \right\rbrace.
\end{equation}
Endowed with the following norm
\begin{equation}
\Vert u \Vert = \left( \int_{\mathbb{R}^{2n}} \frac{\vert u(x) - u(y)\vert^{p}}{\vert x - y \vert^{n+ ps}}dxdy \right)^{\frac{1}{p}}.
\end{equation}
We say that $ u$ is a weak solution to problem \ref{p} if  $ c(x) \vert u \vert^{-\alpha} \phi(x) \in L^{1}(\Omega),$ and for any $ \phi \in X_0,$
\begin{equation}\label{weak}
\begin{aligned}
& M(\Vert u \Vert^{p}) \int_{\Omega} \frac{\vert u(x) - u(y) \vert^{p-2}(u(x) - u(y))(\phi(x) - \phi(y))}{\vert x-y \vert^{n + ps}}\\
& + \int_{\Omega}  (u^+)^{p-2} u \phi(x) dx  - \int_{\Omega} c(x) (u^+)^{-\alpha} \phi(x) - 
\lambda \int_{\Omega} f(x,u^+) \phi(x)dx = 0.
\end{aligned}
\end{equation}
And the energy funcational associated to \ref{p} is $ J_{\lambda} : X_0 \longrightarrow \mathbb{R},$ defined as
\begin{equation}\label{j}
J_{\lambda}(u) = \bar{M}(\Vert u \Vert^{p}) + \frac{1}{p}\int_{\Omega}(u^+)^{p}dx  - \frac{1}{-\alpha+1} \int_{\Omega}c(x)(u^+)^{-\alpha+1} dx - \lambda \int_{\Omega} F(x,u^+) dx .
\end{equation}
where $ \bar{M} =\int_{0}^{t}M(s)ds$, we notice that the funcational $ J_{\lambda}$ is continuous but not differential on $ X_0.$ due the singular term, Now
We define the fiber map $ \phi : \mathbb{R}^+ \longrightarrow \mathbb{R}$ as $ \phi_{u}(t) = J_{\lambda}(tu),$ that is to say
\begin{equation*}
\phi(t) =  \frac{t^P}{p} \left(a \Vert u \Vert^p + \Vert u \Vert^{p}\right) + \frac{b t^{p \theta}}{p \theta}\Vert u \Vert^{p \theta}  -\frac{t^{1-\alpha}}{-\alpha+1} \int c(x)(u^+)^{-\alpha+1} - \lambda t^q \int F(x,u^+) 
\end{equation*}
then we have
\begin{equation}\label{ph1}
\phi^{\prime}(t) = t^{p-1}\left( a \Vert u \Vert^p + \Vert u \Vert_{p}^{p}\right) + b t^{p \theta - 1}\Vert u \Vert^{p \theta}  -t^{-\alpha} \int c(x)(u^+)^{-\alpha+1} - \lambda q t^{q-1} \int F(x,u^+) 
\end{equation}
and 
\begin{equation}\label{ph2}
\begin{aligned}
& \phi^{\prime \prime}(t) = (p-1)t^{p-2} \left(\Vert u \Vert^p + \Vert u \Vert_{p}^{p}\right) + b(p \theta - 1) t^{p \theta - 2}\Vert u \Vert^{p \theta} + \alpha t^{-\alpha -1} \int c(x)(u^+)^{-\alpha+1} \\
&- \lambda q (q-1) t^{q-2} \int F(x,u^+) 
\end{aligned}
\end{equation}
So $ u \in \mathcal{N}_{\lambda} $ if and only if $ \phi^{\prime}_{u}(1)=0,$ it means that elements in $ \mathcal{N}_{\lambda} $ are related to stationary 
points of $ \phi_{u}.$ In order to obtain multiplicity of solutions, we split $\mathcal{N}_{\lambda ,\mu }$  into the following three parts 
\begin{eqnarray*}
\mathcal{N}_{\lambda}^{+} &=&\left\{ u\in \mathcal{N}_{\lambda}:\varphi _{u}^{\prime \prime }(1)>0\right\} =\left\{ u\in X_0:\varphi _{u}^{\prime}(1)=0 \;\mbox{and}\;\varphi _{u}^{\prime \prime }(1)>0\right\}, \\
&& \\
\mathcal{N}_{\lambda}^{-} &=&\left\{ u\in \mathcal{N}_{\lambda }:\varphi _{u}^{\prime \prime }(1)<0\right\}=\left\{ u\in X_0:\varphi _{u}^{\prime}(1)=0 \;\mbox{and}\;\varphi _{u}^{\prime \prime }(1)<0\right\} , \\
&& \\
\mathcal{N}_{\lambda}^{0} &=&\left\{ u\in \mathcal{N}_{\lambda}:\varphi _{u}^{\prime \prime }(1)=0\right\}=\left\{ u\in X_0:\varphi _{u}^{\prime}(1)=0 \;\mbox{and}\;\varphi _{u}^{\prime \prime }(1)=0\right\} .
\end{eqnarray*}
\begin{lemma}\label{l1}
\begin{itemize}
\item Let $ u \in X_+,$ there exist $ \lambda_* > 0,$ and $ t_m, t_1, t_2,$ with $ 0 < t_1 < t_m < t_2,$ such that $ t_1u \in \mathcal{N}_{\lambda}^{+}, t_2u \in \mathcal{N}_{\lambda}^{-},$ for $ \lambda \in (0, \lambda_*).$ Moreover, $ J_{\lambda}(t_1u),$ is the minimum value of $ J_{\lambda}(tu), t \leq t_1,$ and 
$ J_{\lambda}(t_2u),$ is maximum value of  $ J_{\lambda}(tu), t \geq t_m.$
\item If $ u \in X_-,$ and $ \lambda > 0,$ there exist  $ t_3 > 0,$ sch that $ t_3u \in \mathcal{N}_{\lambda}^{+},$ and $ J_{\lambda}(t_3u),$ is the infimum
of $ J_{\lambda}(tu), t > 0.$
\end{itemize} 
\end{lemma}
\begin{proof}
first we introduce the function $ \psi_{u}$ defined as
\begin{equation}
\psi_{u}(t) = t^{p-q}(a \Vert u \Vert^p + \Vert u \Vert_{p}^{p}) + b t^{p \theta - q}\Vert u \Vert^{p \theta} - t^{1 -\alpha -q} \int c(x)(u^+)^{-\alpha+1}
\end{equation}
we allude that $ tu \in N_{\lambda}$ if and only if
\begin{equation*}
\psi_{u}(t) = \lambda q \int F(x,u^+).
\end{equation*}
and we can easly see that
\begin{equation*}
\lim_{t \to 0^+}\psi(t) = -\infty , \quad \lim_{t \to \infty} \psi(t) = 0.
\end{equation*}
Also we have
{\scriptsize
\begin{equation*}
\begin{aligned}
\psi^{\prime}(t) &= (p-q)t^{p-q-1}(a \Vert u \Vert^p + \Vert u \Vert_{p}^{p}) + b (p \theta - q) t^{p \theta - q -1}\Vert u \Vert^{p \theta} - (1 -\alpha -q)t^{ -\alpha -q} \int c(x)(u^+)^{-\alpha+1}\\
&= t^{p \theta - q -1} k(t)
\end{aligned}
\end{equation*}
}
where
{\footnotesize
\begin{equation*}
k(t) = (p-q)t^{p - p \theta}(a \Vert u \Vert^p + \Vert u \Vert_{p}^{p})+ b (p \theta - q) \Vert u \Vert^{p \theta} - (1 -\alpha -q)t^{1 -\alpha -p \theta} \int c(x)(u^+)^{-\alpha+1}
\end{equation*}
}
and since $ \alpha < 1 < p \theta < q \leq  p^*_s,$ we get $ \lim_{t \xrightarrow{>} 0}\psi^{\prime}(t) > 0 \quad \lim_{t \to \infty}\psi^{\prime}(t) < 0$
by simple calculation we obtain that $ k^{\prime}(t_{m}) =0$ for 
\begin{equation*}
t_m = \left( \frac{(p \theta +\alpha -1)(q +\alpha -1)\int c(x)\vert u\vert^{-\alpha+1}}{(p \theta - p)(q-p)(a \Vert u \Vert^p + \Vert u \Vert_{p}^{p})} \right)^{\frac{1}{p+ \alpha -1}} >0.
\end{equation*}
Moreover since $ \alpha < 1 < p \theta < q \leq  p^*_s,$ we have that
\begin{equation*}
\lim_{t \xrightarrow{>} 0}k(t) = +\infty ,\quad \lim_{t \to \infty}k(t) =  b (p \theta - q) \Vert u \Vert^{p \theta} <0.
\end{equation*}
then by the mean value theorem that there exist $ t_{max} >0$ such that $ k(t_{max}) =0,$ which means $ \psi^{\prime}(t_{max}) =0,$ hence $ t_{max}$ is a unique critical point of $ \psi_{u},$ and we know from the fibering analysis that $ \psi_{u}$ is increasing on $ (0,t_{max}),$ decreasing on $ (t_{max},\infty),$ that is to say $ t_{max}$ is a golbal maximum point of $ \psi_{u}.$ We have
\begin{equation*}
\psi_{u}(t) =\bar{\psi(t)} +  b t^{p \theta - q}\Vert u \Vert^{p \theta} 
\end{equation*}
where 
\begin{equation*}
\bar{\psi(t)} =t^{p-q}(a \Vert u \Vert^p + \Vert u \Vert_{p}^{p}) -t^{1 -\alpha -q} \int c(x)(u^+)^{-\alpha+1}
\end{equation*}
We can easly estimate $ \psi_{u}(t_{max})$ through $ \bar{\psi}_{u}(t_{max}),$ hence by \ref{e2} we get
{\footnotesize
\begin{equation}\label{e05}
\begin{aligned}
&\psi_{u}(t_{max}) \geq \max_{t >0}(\bar{\psi}(t)) =\bar{\psi}(t_m)\\
&=\left( a \Vert u \Vert^p + \Vert u \Vert_{p}^{p}\right)^{\frac{q+ \alpha-1}{p+ \alpha-1}}\left(\int c(x)(u^+)^{-\alpha+1}\right)^{\frac{p-q}{p+\alpha -1}}\left(\frac{(q-p)}{q+\alpha-1} \right)^{\frac{q+\alpha-1}{p+\alpha-1}}\left(\frac{p+\alpha -1}{q-p}\right).
\end{aligned}
\end{equation}
}
if  $u \in  X_+$ we get
{\tiny
\begin{equation}
\begin{aligned}
&\max_{t >0} \psi(t) - \lambda q \int F(x,u) \geq \\
&\left( a \Vert u \Vert^p + \Vert u \Vert_{p}^{p}\right)^{\frac{q+ \alpha-1}{p+ \alpha-1}} \left( \Vert c \Vert_{L^{\infty}}\vert \Omega\vert^{\frac{p_s^* +\alpha -1}{p_s^*}}S^{\frac{-\alpha+1}{p}} \Vert u \Vert^{-\alpha+1} \right)^{\frac{p-q}{p+\alpha -1}}\left(\frac{(q-p)}{q+\alpha-1} \right)^{\frac{q+\alpha-1}{p+\alpha-1}}\left(\frac{p+\alpha -1}{q-p}\right)\\
& - \lambda q \int F(x,u)\\ 
&\geq \left( a \Vert u \Vert^p\right)^{\frac{q+ \alpha-1}{p+ \alpha-1}} \left( \Vert c \Vert_{L^{\infty}} \vert \Omega\vert^{\frac{p_s^* +\alpha -1}{p_s^*}}S^{\frac{-\alpha+1}{p}} \Vert u \Vert^{-\alpha+1} \right)^{\frac{p-q}{p+\alpha -1}}\left(\frac{(q-p)}{q+\alpha-1} \right)^{\frac{q+\alpha-1}{p+\alpha-1}}\left(\frac{p+\alpha -1}{q-p}\right)\\
& - \lambda q \gamma \vert \Omega \vert^{\frac{p_s^*-q}{p_s^*}}S^{-\frac{q}{p}} \Vert u \Vert^{q}\\ 
&\geq\Vert u \Vert^q\left( \Vert c \Vert_{L^{\infty}}\vert \Omega\vert^{\frac{p_s^* +\alpha -1}{p_s^*}}S^{\frac{-\alpha+1}{p}}\right)^{\frac{p-q}{p+\alpha -1}}\left(a\frac{(q-p)}{q+\alpha-1} \right)^{\frac{q+\alpha-1}{p+\alpha-1}}\left(\frac{p+\alpha -1}{q-p}\right)\\
& - \lambda q \gamma \vert \Omega \vert^{\frac{p_s^*-q}{p_s^*}}S^{-\frac{q}{p}} \Vert u \Vert^{q}\\ 
&\geq \left(  \vert \Omega \vert^{-\frac{(p_s^*-p)(q +\alpha -1)}{p_s^*(p +\alpha -1)}}S^{\frac{q}{p}}\left( \Vert c \Vert_{L^{\infty}}S^{\frac{-\alpha+1}{p}}\right)^{\frac{p-q}{p+\alpha -1}}\left(a\frac{(q-p)}{q+\alpha-1} \right)^{\frac{q+\alpha-1}{p+\alpha-1}}\left(\frac{p+\alpha -1}{\gamma q (q-p)}\right) - \lambda \right)
 \Vert u \Vert^{q}\\ 
\end{aligned}
\end{equation}
}
if we choose $ \lambda < \lambda_*$ where
\begin{equation*}
\lambda_* :=  \vert \Omega \vert^{-\frac{(p_s^*-p)(q +\alpha -1)}{p_s^*(p +\alpha -1)}}S^{\frac{q}{p}}\left( \Vert c \Vert_{L^{\infty}}S^{\frac{-\alpha+1}{p}}\right)^{\frac{p-q}{p+\alpha -1}}\left(a\frac{q-p}{q+\alpha-1} \right)^{\frac{q+\alpha-1}{p+\alpha-1}}\left(\frac{p+\alpha -1}{\gamma q (q-p)}\right)
\end{equation*}
then there exist $ t_1 < t_m < t_2 $ such that $ \psi(t_1) = \lambda q \int F(x,u) dx = \psi(t_2).$ And we we can see that if $ u \in \mathcal{N}_{\lambda}$ we have $ \phi^{\prime\prime}_{u}(t) =t^{q -1}\psi^{\prime}_{u}(t),$ which yields $ \psi^{\prime}_{u}(t_1) >0$ and $  \psi^{\prime}_{u}(t_2) <0,$ which means that 
$ t_1u \in \mathcal{N}_{\lambda}^{+}$ and $ t_2u \in \mathcal{N}_{\lambda}^{-}.$ And from the fact that $ \phi^{\prime \prime}_{u}(t_1) > 0,$ and $ \phi^{\prime}_{u}(t_1) = 0,$ and $ \phi^{\prime \prime}_{u}(t_2) < 0,$ and $ \phi^{\prime}_{u}(t_2) = 0,$
tells us that $ \phi_{u},$ has a local minimum at $ t_1$ and local maximum at $ t_2.$ That is 
\begin{equation*}
J_{\lambda}(t_1u) =\min_{t \leq t_1}J_{\lambda}(tu), \quad \textit{and} \quad J_{\lambda}(t_2u) =\max_{t \geq t_{max}}J_{\lambda}(tu)
\end{equation*}
if  $u \in  X_-$ have that 
\begin{equation*}
lim_{t \xrightarrow{>} 0}\psi(t) = -\infty \quad  \texttt{and} \quad lim_{t \to \infty}\psi(t) =0.
\end{equation*}
then for all $ \lambda > 0$ we have that
\begin{equation*}
 lim_{t \xrightarrow{>} 0}\psi(t) < \lambda q \int F(x,u) dx \leq  lim_{t \to \infty}\psi(t)
\end{equation*}
hence by the mean value theorem for all $ \lambda > 0,$ there exist $ t_3 \in (0,\infty)$ such that 
\begin{equation*}
\psi(t_3) = \lambda q \int F(x,u) dx \quad \textit{and} \quad  lim_{t \xrightarrow{>} 0}\psi(t) <\psi(t_3)
\end{equation*}
which means that $ \psi$ increasing on $(0, t_3)$ and $ \psi^{\prime}(t_3) > 0.$ hence by $ t^{q-1}\psi^{\prime}(t_3) = \phi^{\prime\prime}(t_3),$ we get hat 
$ t_3 u \in \mathcal{N}_{\lambda}^{+}.$
\end{proof}
\begin{lemma}\label{l2}
there exist $ \lambda_{**} > 0,$ such that for all $ \lambda < \lambda_{**}$ we have $ \mathcal{N}_{\lambda}^{0} = \{0\},$
\end{lemma}
\begin{proof}
Suppose otherwise, let $ u \in \mathcal{N}_{\lambda}^{0},$ such that $ u \neq 0,$ then by \ref{ph1} and \ref{ph2} we have 
\begin{equation} \label{el1}
\phi^{\prime}(1) = \left( a \Vert u \Vert^p + \Vert u \Vert_{p}^{p}\right) + b \Vert u \Vert^{p \theta}  - \int c(x)(u^+)^{-\alpha+1} - \lambda q \int F(x,u^+) = 0 
\end{equation}
\begin{equation} \label{el2}
\begin{aligned}
&\phi^{\prime \prime}(1) = (p-1) \left(a\Vert u \Vert^p + \Vert u \Vert_{p}^{p}\right) + b(p \theta - 1) \Vert u \Vert^{p \theta} + \alpha \int c(x)(u^+)^{-\alpha+1}\\ 
&- \lambda q (q-1) \int F(x,u^+) = 0.
\end{aligned}
\end{equation}
from \ref{el1}, \ref{el2} and using the geometric mean inequality we get
\begin{equation} \label{el3}
\Vert u \Vert \leq  \left( \frac{(q + \alpha -1) \Vert c \Vert _{L^{q^{\prime}}} \vert \Omega \vert^{\frac{p_s^*+ \alpha -1}{p_s^*}} S^{-\frac{-\alpha+1}{p}}}{2 \left( a b (q-p)(q -p\theta) \right)^{\frac{1}{2}}} \right)^{\frac{2}{p \theta + p +2\alpha -2 }}
\end{equation}
Using \ref{el1} and \ref{el3} we get
\footnotesize
\begin{equation}
\begin{aligned}
\phi^{\prime \prime}(1) &= (p-1) \left(a \Vert u \Vert^p + \Vert u \Vert_{p}^{p}\right) + b(p \theta - 1) \Vert u \Vert^{p \theta} + \alpha \left( \left( a \Vert u \Vert^p + \Vert u \Vert_{p}^{p}\right) + b \Vert u \Vert^{p \theta} - \lambda q \int F(x,u^+) \right)\\
&- \lambda q (q-1) \int F(x,u^+)\\
&= (p+ \alpha -1) \left(a \Vert u \Vert^p + \Vert u \Vert_{p}^{p}\right) + b(p \theta + \alpha - 1) \Vert u \Vert^{p \theta} -\lambda q (q + \alpha -1   ) \int F(x,u^+)\\
& \geq a (p+ \alpha -1) \Vert u \Vert^p + b(p \theta + \alpha - 1) \Vert u \Vert^{p \theta} -\lambda q (q + \alpha -1   )  \gamma \vert \Omega \vert^{\frac{p_s^*-q}{p_s^*}}S^{-\frac{q}{p}} \Vert u \Vert^{q}\\
& \geq  2\left( a b (p+ \alpha -1)(p \theta + \alpha - 1)\right)^{\frac{1}{2}}\Vert u \Vert^{\frac{p+p \theta}{2}}-\lambda q (q + \alpha -1   )  \gamma \vert \Omega \vert^{\frac{p_s^*-q}{p_s^*}}S^{-\frac{q}{p}} \Vert u \Vert^{q}\\
& \geq \left( 2\left( a b (p+ \alpha -1)(p \theta + \alpha - 1)\right)^{\frac{1}{2}} -\lambda q (q + \alpha -1   )  \gamma \vert \Omega \vert^{\frac{p_s^*-q}{p_s^*}}S^{-\frac{q}{p}}\Vert u \Vert^{\frac{2q - p \theta - p}{2}} \right) \Vert u \Vert^{\frac{p+p \theta}{2}}\\
& \geq \left( \frac{2\left(a b (p+ \alpha -1)(p \theta + \alpha - 1)\right)^{\frac{1}{2}}}{q (q + \alpha -1   )  \gamma \vert \Omega \vert^{\frac{p_s^*-q}{p_s^*}}S^{-\frac{q}{p}}}  -\lambda \Vert u \Vert^{\frac{2q - p \theta - p}{2}} \right) \Vert u \Vert^{\frac{p+p \theta}{2}}\\
& \geq \left(\lambda_{**} - \lambda \right)\Vert u \Vert^{\frac{p+p \theta}{2}} > 0.\\
\end{aligned}
\end{equation}
Where 
{\scriptsize
\begin{equation}
\lambda_{**} = \left(\frac{4 ab}{(q + \alpha -1)^2}\right)^{\frac{q +\alpha -1}{p\theta +p +2\alpha -2}}\frac{\left((p+ \alpha -1)(p \theta + \alpha - 1)\right)^{\frac{1}{2}}S^{\frac{(\theta +1)(q +\alpha -1)}{p\theta +p +2\alpha -2}}
\vert \Omega \vert^{-\frac{(q+ \alpha -1)(2p^*_s -p\theta -p)}{p^*_s (p\theta +p +2\alpha -2)}}}{\left((q-p)(q -p\theta)\right)^{-\frac{1}{2}\frac{2q - p \theta - p}{p \theta + p +2\alpha -2 }} q\gamma \Vert c \Vert _{L^{q^{\prime}}}^{\frac{2q - p \theta - p}{p \theta + p +2\alpha -2 }}}
\end{equation}
}
Since $ p\theta < q \leq p^*_s,$ then if we choose $ \lambda < \lambda_{**},$ we obtain that $ \phi^{\prime \prime}(1) > 0,$ which is our contradiction,
\end{proof}
We use the following lemma to show that $ \mathcal{N}^{-}_{\lambda},$  is closed set in the topology of the space $ X_0,$ 
\begin{lemma} \label{l3}
There exist a gap structure in $ \mathcal{N}_{\lambda},$ such that for $ \lambda < \lambda_{**},$ we have
\begin{equation*}
\Vert u \Vert < \eta_{0} < \eta_{\lambda} < \Vert v \Vert, \quad \textit{for any} \quad u \in \mathcal{N}_{\lambda}^{+},\,\, v \in \mathcal{N}_{\lambda}^{-}.
\end{equation*}
Where
{\tiny
\begin{equation*}
\eta_0 = \left( \frac{(q +\alpha -1) \Vert c\Vert_{q^{\prime}} S^{-\frac{-\alpha +1}{p}} \vert \Omega\vert^{\frac{p^*_s  +\alpha -1}{p^*_s}}}{2
\left(ab(q-p)(q-p\theta) \right)^{\frac{1}{2}}} \right)^{\frac{2}{p\theta +p +2\alpha -2}}.  \eta_{\lambda} = \left( \frac{2 \left(ab(p +\alpha -1)(p\theta +\alpha -1)\right)^{\frac{1}{2}}}{\lambda q(q+ \alpha -1) S^{-\frac{q}{p}} \vert \Omega\vert^{\frac{p_s^*-q}{p_s^*}} } \right)^{\frac{2}{2q -\theta p -p}}.
\end{equation*}
}
\end{lemma}
\begin{proof}
If $ u \in \mathcal{N}_{\lambda},$ then by \ref{el1} we get
\begin{equation} \label{e10}
\phi^{\prime \prime}(1) = (p +\alpha -1) \left(a\Vert u \Vert^p + \Vert u \Vert_{p}^{p}\right) + b(p \theta +\alpha -1)\Vert u \Vert^{p \theta} - \lambda q (q +\alpha -1) \int F(x,u^+) .
\end{equation}
\begin{equation} \label{e11}
\phi^{\prime \prime}(1) = (p-q) \left(a\Vert u \Vert^p + \Vert u \Vert_{p}^{p}\right) + b(p \theta - q) \Vert u \Vert^{p \theta} + (q +\alpha -1)\int c(x)( u^+)^{-\alpha+1}.
\end{equation}
Let $ u \in \mathcal{N}_{\lambda}^{+},$ then Using \ref{e11} and the geometric mean inequality we get
\begin{equation*}
\Vert u \Vert < \left( \frac{(q +\alpha -1) \Vert c\Vert_{q^{\prime}} S^{-\frac{-\alpha +1}{p}} \vert \Omega\vert^{\frac{p^*_s  +\alpha -1}{p^*_s}}}{2
\left(ab(q-p)(q-p\theta) \right)^{\frac{1}{2}}} \right)^{\frac{2}{p\theta +p +2\alpha -2}} = \eta_0.
\end{equation*}
If $ v \in \mathcal{N}_{\lambda}^{-},$ then Using \ref{e10} and the geometric mean inequality we find
\begin{equation*}
\Vert v \Vert > \left( \frac{2 \left(ab(p +\alpha -1)(p\theta +\alpha -1)\right)^{\frac{1}{2}}}{\lambda q(q+ \alpha -1) S^{-\frac{q}{p}} \vert \Omega\vert^{\frac{p_s^*-q}{p_s^*}} } \right)^{\frac{2}{2q -\theta p -p}} = \eta_{\lambda}.
\end{equation*}
Now, we show that $ \eta_{0} < \eta_{\lambda},$ for $ \lambda < \lambda_{**},$ by simple calculation we get 
\begin{equation}
\eta_0 - \eta_{\lambda} =\left(\frac{\left( (q +\alpha -1) \Vert u\Vert_{q^{\prime}} S_{p}^{\frac{\alpha -1}{p}} \vert \Omega\vert^{\frac{p^*_s +\alpha -1}{p}} \right)^{2}}{(4 ab) (q-p)(q-p\theta)} \right)^{\frac{1}{p\theta +p +2\alpha -2}}\left( 1 -\gamma^{\frac{2}{2q -p\theta -p}} \right).
\end{equation} 
Hence and \ref{e1} our result follows.
\end{proof}
\begin{corollary}\label{cl3}
Now, let $ (u_n)_n,$ be a sequence in $ \mathcal{N}_{\lambda}^{-},$ where $ u_n \longrightarrow u$ in $ X_0.$ then $ u \in \overline{\mathcal{N}_{\lambda}^{-}},$ 
and by lemma \ref{l3}, we have $ \Vert u\Vert = \lim_{n \to \infty} \Vert u_n\Vert \geq \eta_{\lambda} > \eta_{0} >0.$ which means that $ u \neq 0.$ Hence 
$ \overline{\mathcal{N}_{\lambda}^{-}} = \mathcal{N}_{\lambda}^{-}.$
\end{corollary}
\begin{lemma}\label{l4}
Let $ u \in \mathcal{N}_{\lambda}^{-},$ and $ \lambda > 0.$ Then there exist $ \epsilon > 0$ and a continuous function 
$ \zeta : B_{\epsilon}(0) \longrightarrow \mathbb{R}^{+}$ such that
\begin{equation*}
\zeta(\phi) > 0, \zeta(0) = 1, \,\,\, \zeta(\phi)(u +\phi)\,\,\, in\,\,\, \mathcal{N}_{\lambda}^{-} \,\,\, \textit{for all} \,\,\, \phi \,\,\, \in \,\,\, B_{\epsilon}(0).
\end{equation*}
Where  $ B_{\epsilon}$ is the open ball in $ X_0$ centered at $ 0$ and of radius $ \epsilon.$
\end{lemma}
\begin{proof}
Let $ H : X_0 \times \mathbb{R}^{+} \longrightarrow \mathbb{R}$ defined as
{\footnotesize
\begin{equation*}
\begin{aligned}
& H(\phi, s) =s^{p +\alpha -1} \left( a\Vert u +\phi\Vert^{p} +\Vert u +\phi\Vert^{p}_{p}\right) +b s^{p\theta +\alpha -1}\Vert u +\phi\Vert^{p\theta} -\int_{\Omega}c(x)((u +\phi)^+)^{1 -\alpha}\\
&- \lambda q s^{q +\alpha -1} \int_{\Omega} F(x,u +\phi)^+)
\end{aligned}
\end{equation*}
} 
Since $ u \in \mathcal{N}_{\lambda}^{-}$ and \ref{e10} we get
\begin{equation}\label{e13}
H(0, 1) = \left( a\Vert u\Vert^{p} +\Vert u\Vert^{p}_{p}\right) +b \Vert u\Vert^{p\theta}-\int_{\Omega}c(x)(u^+)^{1 -\alpha} - \lambda q\int_{\Omega} F(x,u^+) =0.
\end{equation}
{\footnotesize
\begin{equation} \label{e14}
\frac{\partial H}{\partial s}(0, 1) =(p +\alpha -1)\left( a\Vert u\Vert^{p} +\Vert u\Vert^{p}_{p}\right) +b (p\theta +\alpha -1)\Vert u\Vert^{p\theta} -\lambda q
(q +\alpha -1)\int_{\Omega} F(x,u^+) < 0.
\end{equation}
}
Therefore by the implicit function theorem at the point $ (0, 1)$ there exist $ \delta >0$ such that for any $ \phi \in B_{X_0}(0, \delta),$ the equation
$ H(\phi, s) =0$ has a unique stricly positive continuous solution $ s = \zeta(\phi),$ and $ \zeta(0) =1.$ Since $ H(\phi, \zeta(\phi)) =0$ for 
$ \phi \in B_{X_0}(0, \delta),$ we obtain
\begin{equation*}
\begin{aligned}
&0 =H(\phi, \zeta(\phi)) =(\zeta(\phi))^{p +\alpha -1} \left( a\Vert u +\phi\Vert^{p} +\Vert u +\phi\Vert^{p}_{p}\right) +b (\zeta(\phi))^{p\theta +\alpha -1}\Vert u +\phi\Vert^{p\theta} \\
& -\int_{\Omega}c(x)((u +\phi)^+)^{1 -\alpha} - \lambda q (\zeta(\phi))^{q +\alpha -1} \int_{\Omega} F(x,(u +\phi)^+)\\
&=(\zeta(\phi))^{\alpha -1} \left( a (\zeta(\phi)\Vert u +\phi\Vert)^{p} +(\zeta(\phi)\Vert u +\phi\Vert)^{p}_{p}\right) +b (\zeta(\phi))^{\alpha -1}(\zeta(\phi) \Vert u +\phi\Vert)^{p\theta} \\
& -\zeta(\phi)^{\alpha -1}\int_{\Omega} c(x) (\zeta(\phi)((u +\phi)^+)^{1 -\alpha} - \lambda q (\zeta(\phi))^{\alpha -1} \int_{\Omega} F(x,\zeta(\phi)(u +\phi)^+)\\
&=(\zeta(\phi))^{\alpha -1} \left( a\Vert \zeta(\phi)(u +\phi)\Vert^{p} +\Vert\zeta(\phi)(u +\phi)\Vert^{p}_{p}\right) +b \zeta(\phi)^{\alpha -1}\Vert \zeta(\phi)(u +\phi)\Vert^{p\theta} \\
& -\zeta(\phi)^{\alpha -1}\int_{\Omega} c(x)(\zeta(\phi)(u +\phi)^+))^{1 -\alpha} - \lambda q \zeta(\phi)^{\alpha -1} \int_{\Omega} F(x,\zeta(\phi)(u +\phi)^+)\\
\end{aligned}
\end{equation*}
That is to say $ \zeta(\phi)(u +\phi) \in \mathcal{N}_{\lambda}$ for $ \phi \in B_{X_0}(0, \delta).$ Also we have
\footnotesize
\begin{equation*}
\begin{aligned}
& \frac{\partial H}{\partial s}(\phi, \zeta(\phi))  = (p +\alpha -1)\zeta(\phi)^{\alpha -2} \left( a\Vert \zeta(\phi)(u +\phi)\Vert^{p} +\Vert \zeta(\phi)(u +\phi)\Vert^{p}_{p}\right)\\
& +b (p\theta +\alpha -1) \zeta(\phi)^{\alpha -2}\Vert \zeta(\phi)(u +\phi)\Vert^{p\theta} -\lambda q (q +\alpha -1) \zeta(\phi)^{\alpha -2} \int_{\Omega} F(x,\zeta(\phi)(u +\phi)^+)dx
\end{aligned}
\end{equation*} 
Therefore as in \ref{e14} we can choose $ 0 <\epsilon \leq \delta,$ small enough such that for any $ \phi \in B_{X_0}(0, \epsilon)$ we find
\footnotesize
\begin{equation*}
\begin{aligned}
&(p +\alpha -1)\left( a\Vert \zeta(\phi)(u +\phi)\Vert^{p} +\Vert \zeta(\phi)(u +\phi)\Vert^{p}_{p}\right) +b (p\theta +\alpha -1) \Vert \zeta(\phi)(u +\phi)\Vert^{p\theta}\\
&- \lambda q (q +\alpha -1) \int_{\Omega} F(x,\zeta(\phi)(u +\phi))dx < 0.
\end{aligned}
\end{equation*} 
Which means that $ \zeta(\phi)(u +\phi) \in \mathcal{N}_{\lambda}^{-}$ for all $ \phi \in B_{X_0}(0, \epsilon).$\newline
The proof for $ \mathcal{N}_{\lambda}^{+}$ is analogous.
\end{proof}
\begin{lemma} \label{l5}
The funcational $ J_{\lambda}$ is coercive and bounded from bellow on $ \mathcal{N}_{\lambda}.$ 
\end{lemma}
\begin{proof}
Let $ u \in \mathcal{N}_{\lambda},$ then by \ref{e1} we have
\footnotesize
\begin{equation*}
\begin{aligned}
&J(u) =J(u) -\frac{1}{p\theta} \langle J_{\lambda}^{\prime}, u \rangle\\
     &= \left( \frac{1}{p} -\frac{1}{p\theta}\right)\left(a \Vert u \Vert^p + \Vert u \Vert^{p}_p\right) -\left(\frac{1}{1 -\alpha} - \frac{1}{p\theta} \right)\int c(x)(u^+)^{-\alpha+1} -\lambda \left( 1 -\frac{q}{p\theta}\right)\int F(x,u^+)dx.\\
     &\geq  a\left( \frac{1}{p} -\frac{1}{p\theta}\right)\Vert u \Vert^p -\left(\frac{1}{1 -\alpha} - \frac{1}{p\theta} \right)\int c(x)(u^+)^{-\alpha+1}\\
     &\geq  a\left( \frac{1}{p} -\frac{1}{p\theta}\right)\Vert u \Vert^p -\left(\frac{1}{1 -\alpha} - \frac{1}{p\theta} \right)\left( \Vert c \Vert_{L^{q^{\prime}}}\vert \Omega \vert^{\frac{p_s^* +\alpha -1}{p_s^*}}S^{\frac{-\alpha+1}{p}} \Vert u \Vert^{-\alpha+1}\right)\\
     &\geq  a\left( \frac{1}{p} -\frac{1}{p\theta}\right)\Vert u \Vert^p -\Vert c \Vert_{L^{q^{\prime}}}\left(\frac{1}{1 -\alpha} - \frac{1}{p\theta} \right)\vert \Omega \vert^{\frac{p_s^* +\alpha -1}{p_s^*}}S^{\frac{-\alpha+1}{p}} \Vert u \Vert^{1 -\alpha}\\
     &\geq C
\end{aligned}
\end{equation*}
Since $ 1 -\alpha < p < p\theta,$ we get that the functional $ J_{\lambda}$ is coercive. Moreover
{\scriptsize
\begin{equation*}
\begin{aligned}
&C =-\left(\frac{1}{1 -\alpha} - \frac{1}{p\theta} \right)^{\frac{p}{p +\alpha +1}}\Vert c \Vert_{L^{q^{\prime}}}^{\frac{p}{p +\alpha+1}}\vert \Omega \vert^{\frac{p}{p_s^*}\frac{p_s^* +\alpha -1}{p +\alpha -1}}S^{\frac{1 -\alpha}{p +\alpha +1}}\left( a\left( \frac{1}{p} -\frac{1}{p\theta}\right) \right)^{\frac{\alpha -1}{p +\alpha -1}}\\
&\left(\frac{1 -\alpha}{p}\right)^{\frac{p}{p +\alpha -1}}\left(\frac{p +\alpha -1}{1 -\alpha}\right) < 0.
\end{aligned}
\end{equation*}
}
where $ C$ is the minimum of the function $ h$ defined as 
\begin{equation*}
h(s) =a\left( \frac{1}{p} -\frac{1}{p\theta}\right)s^{\frac{p}{1 -\alpha}} -\Vert c \Vert_{L^{q^{\prime}}}\left(\frac{1}{1 -\alpha} - \frac{1}{p\theta} \right)\vert \Omega \vert^{\frac{p_s^* +\alpha -1}{p_s^*}}S^{\frac{-\alpha+1}{p}}s,
\end{equation*}
Attained at $ s_m$ given by
\begin{equation*}
s_m =\left( \frac{\Vert c \Vert_{L^{q^{\prime}}}\left(\frac{1}{1 -\alpha} - \frac{1}{p\theta} \right)\vert \Omega \vert^{\frac{p_s^* +\alpha -1}{p_s^*}}S^{\frac{-\alpha+1}{p}}}{\frac{a p}{1 -\alpha}\left( \frac{1}{p} -\frac{1}{p\theta}\right)} \right)^{\frac{1- \alpha}{p +\alpha -1}}.
\end{equation*}
\end{proof}
\section{main results}
\begin{corollary}\label{c1}
let $ u \in \mathcal{N}_{\lambda}^{+},$ then by \ref{el1}, \ref{e10} and \ref{e11} we get
\footnotesize
\begin{equation}
\begin{aligned}
J(u) &=  \frac{1}{p} \left(a \Vert u \Vert^p + \Vert u \Vert^{p}_p\right) + \frac{b}{p \theta}\Vert u \Vert^{p \theta}  -\frac{1}{-\alpha+1} \int c(x)( u^+)^{-\alpha+1} - \lambda \int F(x,u^+)dx\\
     &= \left( \frac{1}{p} -\frac{1}{q}\right)\left( a \Vert u \Vert^p + \Vert u \Vert_{p}^{p}\right) +b\left( \frac{1}{p\theta} -\frac{1}{q}\right)\Vert u \Vert^{p \theta} -\left(\frac{1}{1-\alpha} -\frac{1}{q}\right)\int c(x)(u^+)^{-\alpha+1}\\
     &\leq  -\frac{(q-p)(p +\alpha -1)}{q p (1 -\alpha)}\left( a \Vert u \Vert^p + \Vert u \Vert^{p}\right) -b\frac{(q -p\theta)(p\theta +\alpha -1)}{q p\theta (1 -\alpha)}\Vert u \Vert^{p \theta} < 0
\end{aligned}
\end{equation}
\end{corollary}
\begin{corollary}\label{ccr}
From lemma \ref{l2} and lemma \ref{l5} for all $ \lambda < \lambda_{**},$ we can define
\begin{equation*}
m^{-} =\inf_{u \in \mathcal{N}_{\lambda}^{-}} J_{\lambda}(u), \quad m^{+} =\inf_{u \in \mathcal{N}_{\lambda}^{+} \cup \{0\}} J_{\lambda}(u). 
\end{equation*} 
We have by lemma \ref{l2}, \ref{l3} that $ \mathcal{N}_{\lambda}^{0} = \{ 0\},$ and $ \mathcal{N}_{\lambda}^{-}$ is a closed in $ X_0,$ so to find the infimum of 
$ J_{\lambda}$ on $ \mathcal{N}_{\lambda}^{-}$ and on $ \mathcal{N}_{\lambda}^{+} \cup \{ 0\}.$ We use the Ekeland variational principle, let $ (u_n)_n \in \mathcal{N}_{\lambda}^{-}, \mathcal{N}_{\lambda}^{+} \cup \{ 0\},$ be a minimizing sequence for $ J_{\lambda},$ then
\begin{equation}\label{ek1}
m^{-} < J_{\lambda}(u_n) < m^{-} + \frac{1}{n},\quad \textit{and} \quad  m^{+} < J_{\lambda}(u_n) < m^{+} + \frac{1}{n},
\end{equation}
and
\begin{equation}\label{ek2}
J_{\lambda}(u) \geq J_{\lambda}(u_n) -\frac{1}{n} \Vert u -u_n \Vert, \quad \textit{for all} \quad u \in \mathcal{N}_{\lambda},
\end{equation}
and by \ref{ek2} we see that $ J_{\lambda}(u_n) \longrightarrow \inf_{u \in \mathcal{N}_{\lambda}^{+} \cup \{0\}} J_{\lambda}(u)$ as $ n \to \infty.$  
By corollary \ref{c1} and the SLC of norm
\begin{equation}\label{slc}
J_{\lambda}(u_*) \leq \lim_{n \to \infty} \inf_{u \in \mathcal{N}_{\lambda}^{+} \cup \{0\}}J_{\lambda}(u_n) =m_+ < 0.
\end{equation}
Due to the coercivity of $ J_{\lambda}$ the sequence $ (u_n)_n$ is bounded in $ X_0,$ Hence up to subsequence denoted $ (u_n)_n$ we have
\begin{itemize}\label{v02}
\item  $ u_n \rightharpoonup u_*$ \,\, \textit{converges weakly to some} \,\, $ u_* \in X_0,$
\item  $ u_n \longrightarrow u_*$ \,\, \textit{ strongly in} \,\, $ L^{\eta}(\Omega)$ \,\, \textit{for}\,\,$ \eta \in [1, p_s^*),$
\item  $ u_n(x) \longrightarrow u_*(x)$ \,\, \textit{a.e. in} \,\, $ \Omega$
\end{itemize}
As $ n \to \infty.$
\end{corollary}
\begin{lemma}\label{l6}
Let $ \lambda \in (0, \lambda_*), (u_n)_n \in \mathcal{N}$ we have the following inequalities hold for any $ n \in \mathbb{N},$
\begin{itemize}
\item If $ (u_n)_n \in \mathcal{N}_{\lambda}^{+},$ then
\begin{equation*}
(p +\alpha -1) \left(a\Vert u_n \Vert^p + \Vert u_n \Vert_{p}^{p}\right) + b(p \theta +\alpha -1)\Vert u_n \Vert^{p \theta} - \lambda q (q +\alpha -1) \int F(x,u_n^+) >0.
\end{equation*}
\item If $ (u_n)_n \in \mathcal{N}_{\lambda}^{-},$ then
\begin{equation*}
(p +\alpha -1) \left(a\Vert u_n \Vert^p + \Vert u_n \Vert_{p}^{p}\right) + b(p \theta +\alpha -1)\Vert u_n \Vert^{p \theta} - \lambda q (q +\alpha -1) \int F(x,u_n^+) <0.
\end{equation*}
\end{itemize}
\end{lemma}
\begin{proof}
From \ref{el1} and \ref{el2} it is equivalent to prove the following inequality for any $ n \in \mathbb{N}$ 
\begin{equation}\label{e01}
(q -p) \left(a\Vert u_n \Vert^p + \Vert u_n \Vert_{p}^{p}\right) + b(q -p \theta) \Vert u_n \Vert^{p \theta} <(q +\alpha -1)\int c(x)(u_n^+)^{-\alpha+1}
\end{equation}
we will prove \ref{e01} by contradiction suppose that
\begin{equation}\label{e03}
\lim_{n \to \infty}(q-p) \left(a\Vert u_n \Vert^p + \Vert u_n \Vert_{p}^{p}\right) + b(q -p\theta) \Vert u_n \Vert^{p \theta} \geq (q +\alpha -1)\int c(x)( u_*^+)^{-\alpha+1} 
\end{equation}
Since $ u_n \in \mathcal{N}_{\lambda}^{+},$ for $ n \in \mathbb{N}$ we get 
{\scriptsize
\begin{equation*}
\lim_{n \to \infty}\inf\left((q-p) \left(a\Vert u_n \Vert^p + \Vert u_n \Vert_{p}^{p}\right) + b(q -p\theta) \Vert u_n \Vert^{p \theta}\right) < \lim_{n \to \infty}\sup(q +\alpha -1)\int c(x)(u_n^+)^{-\alpha+1}
\end{equation*}
}
Then by \ref{e02}
\begin{equation}\label{e04}
\lim_{n \to \infty}(q-p) \left(a\Vert u_n \Vert^p + \Vert u_* \Vert_{p}^{p}\right) + b(q -p\theta) \Vert u_n \Vert^{p \theta} \leq (q +\alpha -1)\int c(x)( u_*^+)^{-\alpha+1}.
\end{equation}
hence by combining \ref{e03} and \ref{e04} we get 
\begin{equation*}
\lim_{n \to \infty}(q-p) \left(a\Vert u_n \Vert^p + \Vert u_* \Vert_{p}^{p}\right) + b(q -p\theta) \Vert u_n \Vert^{p \theta} =(q +\alpha -1)\int c(x)( u_*^+)^{-\alpha+1}.
\end{equation*}
then there exist $ \beta >0$ such that $ \Vert u_n \Vert^p \longrightarrow \beta.$ Therefore
\begin{equation}\label{e06}
(q-p) \left(a\beta + \Vert u_* \Vert_{p}^{p}\right) + b(q -p\theta) \beta^{\theta} =(q +\alpha -1)\int c(x)(u_*^+)^{-\alpha+1}.
\end{equation}
We know frome lemma \ref{l1} for any $ \lambda \in (0, \lambda_*)$ that
\begin{equation}\label{e07}
\bar{\psi}(t_m) -\lambda\int_{\Omega}F(x,u_n^+)dx >0,
\end{equation}
Therefore by \ref{e06}  we find 
\footnotesize
\begin{equation}
\begin{aligned}
 &0 < \left( a \Vert u_n \Vert^p + \Vert u_n \Vert_{p}^{p}\right)^{\frac{q+ \alpha-1}{p+ \alpha-1}}\left(\int c(x)(u_n^+)^{-\alpha+1}\right)^{\frac{p-q}{p+\alpha -1}}\left(\frac{(q-p)}{q+\alpha-1} \right)^{\frac{q+\alpha-1}{p+\alpha-1}}\left(\frac{p+\alpha -1}{q-p}\right)\\ 
& -\lambda\int_{\Omega}F(x,u_n^+)dx\\
 0 &< \left( a \Vert u_n \Vert^p + \Vert u_n \Vert_{p}^{p}\right)^{\frac{q+ \alpha-1}{p+ \alpha-1}}\left(\int c(x)(u_n^+)^{-\alpha+1}\right)^{\frac{p-q}{p+\alpha -1}}\left(\frac{(q-p)}{q+\alpha-1} \right)^{\frac{q+\alpha-1}{p+\alpha-1}}\left(\frac{p+\alpha -1}{q-p}\right)\\
 & -\frac{1}{q}\left(\left( a \Vert u_n \Vert^p + \Vert u_n \Vert_{p}^{p}\right) + b \Vert u_n \Vert^{p \theta}  - \int c(x)(u_n^+)^{-\alpha+1}\right)\\
 0 &\leq \left( a\beta + \Vert u_* \Vert_{p}^{p}\right)^{\frac{q+ \alpha-1}{p+ \alpha-1}}\left(\frac{(q-p) \left(a\beta + \Vert u_* \Vert_{p}^{p}\right) + b(q -p\theta) \beta^{\theta}}{q +\alpha -1}\right)^{\frac{p-q}{p+\alpha -1}}\left(\frac{(q-p)}{q+\alpha-1} \right)^{\frac{q+\alpha-1}{p+\alpha-1}}\\
 &\left(\frac{p+\alpha -1}{q-p}\right) -\frac{1}{q}\left(\left( a\beta + \Vert u_* \Vert_{p}^{p}\right) + b \beta^{\theta}  - \frac{(q-p) \left(a\beta + \Vert u_* \Vert_{p}^{p}\right) + b(q -p\theta) \beta^{\theta}}{q +\alpha -1}\right)\\  
  &\leq  -b\frac{p\theta +\alpha -1}{q+ \alpha -1}\beta^{\theta} < 0.\\ 
\end{aligned}
\end{equation}
Which is contraduction with \ref{e07} since $ \alpha <1 <p\theta<q.$
\end{proof}
\begin{lemma}\label{l7}
Let $ (u_n)_n \in \mathcal{N}_{\lambda}^{-},$ (resp $ \mathcal{N}_{\lambda}^{+},$) such that up to subsequence  $ u_n \rightharpoonup u_*$ on $ X_0$ and satidfying 
\ref{ek2}, then for $ \lambda \in (0, \lambda_*)$ we have, $ \langle\zeta^{\prime}(0), \Phi \rangle$ is uniformly bounded for any positive $ \Phi$ in $ X_0.$
\end{lemma}
\begin{proof}
Let $ u_n \in \mathcal{N}_{\lambda}^{+},$ then we have
\begin{equation}\label{e08}
\left( a \Vert u_n \Vert^p + \Vert u_n \Vert_{p}^{p}\right) + b \Vert u_n \Vert^{p \theta}  - \int c(x)(u_n^+)^{-\alpha+1} - \lambda q \int F(x,u_n^+) = 0 
\end{equation}
and by lemma \ref{l4}, for $ \Phi \in X_0,$ where $ \Phi \geq 0,$ there exist a sequence of functions $ (\zeta_n)_n$ such that 
$ \zeta_n(t\Phi(u_n +t\Phi)) \in \mathcal{N}_{\lambda}^{+}$ and $ \zeta_n(0) =1.$ which yields
\begin{equation}\label{e09}
\begin{split}
\zeta_n^{p}(t\Phi) \left( a \Vert u_n +t\Phi \Vert^p + \Vert u_n +t\Phi \Vert_{p}^{p}\right) + b \zeta_n^{p\theta}(t\Phi) \Vert u_n +t\Phi \Vert^{p \theta}\\ 
-\zeta_n^{-\alpha+1}(t\Phi) \int c(x)((u_n +t\Phi)^+)^{-\alpha+1} - \lambda \zeta_n^{q}(t\Phi) q \int F(x,(u_n +t\Phi)^+) = 0 
\end{split}
\end{equation} 
Where $ \langle\zeta_n^{\prime}(0), \Phi\rangle \in \bar{\mathbb{R}},$ for any $ \Phi \in X_0,$ assuming that the right derivative of $ \zeta$ at $ t =0$ exists,
if not we take $ t_n >0$ such that $ t_n \xrightarrow{>} 0,$ and then we take $ \zeta_n^{\prime}(0)$ as $ \lim_{n \to \infty} s_n =\frac{\zeta_n(t_n\Phi) -1}{t_n}.$ Therefore by \ref{e08} and \ref{e09} we obtain
\begin{equation}
\begin{aligned}
&(\zeta_n^{p}(t\Phi) -1) \left(a\Vert u_n +t\Phi \Vert^p + \Vert u_n +t\Phi \Vert_{p}^{p}\right)+a\left(\Vert u_n +t\Phi \Vert^p -\Vert u_n \Vert^p\right)\\ 
& +\left(\Vert u_n +t\Phi \Vert_{p}^{p}-\Vert u_n \Vert_{p}^{p}\right) +b (\zeta_n^{p\theta}(t\Phi) -1) \Vert u_n +t\Phi \Vert^{p \theta}  +b\left(\Vert u_n +t\Phi \Vert^{p \theta} -\Vert u_n \Vert^{p \theta}\right)\\
&-(\zeta_n^{-\alpha+1}(t\Phi) -1)\int c(x)((u_n +t\Phi)^+)^{-\alpha+1}\\
&-\left(\int c(x)((u_n +t\Phi)^+)^{-\alpha+1} -\int c(x)\vert u_n\vert^{-\alpha+1}\right)\\
&-\lambda q (\zeta_n^{q}(t\Phi) -1) \int F(x,(u_n +t\Phi)^+)\\
& -\lambda q \left(\int F(x,(u_n +t\Phi)^+) -\int F(x,u_n^+)\right) =0\\
\end{aligned}
\end{equation} 
deviding by $ t >0$ and passing to the limit when $ t \xrightarrow{>} 0,$ and usign \ref{e08} we find,
\begin{equation}
\begin{aligned}
& 0\leq p\zeta^{p-1}_{n}(t\Phi)\langle \zeta^{\prime}_{n}(0), \Phi\rangle\left( a \Vert u_n \Vert^p + \Vert u_n \Vert_{p}^{p}\right)\\
&+pa\int_{\Omega}\frac{\vert u_n(x) -u_n(y)\vert^{p-2}(u_n(x) -u_n(y))(\Phi(x) -\Phi(y))}{\vert x -y\vert^{n  +ps}}dxdy \\
&+\int_{\Omega}(u_n^+)^{p -1}\Phi dx +b p\theta\zeta^{p\theta-1}_{n}(t\Phi)\langle \zeta^{\prime}_{n}(0), \Phi\rangle\Vert u_n \Vert^{p\theta} 
+b p\theta\Vert u_n\Vert^{p(\theta -1)}\\
&\int_{\Omega}\frac{\vert u_n(x) -u_n(y)\vert^{p-2}(u_n(x) -u_n(y))(\Phi(x) -\Phi(y))}{\vert x -y\vert^{n  +ps}}dxdy\\
&-(-\alpha +1)\zeta^{-\alpha}_{n}(t\Phi)\langle \zeta^{\prime}_{n}(0), \Phi\rangle\int c(x)(u_n^+)^{-\alpha+1}\\
& -(-\alpha +1)\int_{\Omega}c(x)(u_n^+)^{-\alpha}\Phi -\lambda q^{2}\zeta^{q-1}_{n}(t\Phi)\langle \zeta^{\prime}_{n}(0), \Phi\rangle\int_{\Omega}F(x,u_n^+)dx\\
&-\lambda q\int_{\Omega}f(x,u_n^+)\Phi dx\\
&= \langle \zeta^{\prime}_{n}(0), \Phi\rangle\\
&\left( (p +\alpha -1)\left( a \Vert u_n \Vert^p + \Vert u_n \Vert_{p}^{p}\right) +b(p\theta +\alpha -1) -\lambda q(q +\alpha -1)\int F(x,u_n^+)dx\right)\\
& +\left(a p +b p\theta\Vert u_n\Vert^{p(\theta -1)}\right)\int_{\Omega}\frac{\vert u_n(x) -u_n(y)\vert^{p-2}(u_n(x) -u_n(y))(\Phi(x) -\Phi(y))}{\vert x -y\vert^{n  +ps}}dxdy \\
&-\lambda q\int_{\Omega}f(x,u_n^+)\Phi dx +\int_{\Omega}(u_n^+)^{p -1}\Phi dx
\end{aligned}
\end{equation}
Since $ u_n$ is bounded in $ X_0$ and by lemma \ref{l6} there exist a constant 
$ C_1 >0$ and $ C_2 \in \mathbb{R}$ such that $ \langle \zeta^{\prime}_{n}(0) \Phi\rangle \geq \frac{C_2}{C_1},$ 
hence $ \langle \zeta^{\prime}_{n}(0) \Phi\rangle$ is bounded from bellow for any positive $ \Phi$ in $ X_0.$\newline
On the other hand using \ref{ek2} with $ u =\zeta_n(t\Phi)(u_n +t\Phi),$ we obtain
\begin{equation}
\begin{aligned}
& \vert (\zeta_n(t\Phi) -1)\frac{\Vert u_n\Vert}{n} +\zeta_n(t\Phi)\frac{\Vert t\Phi\Vert}{n} 
\geq J_{\lambda}(u_n) -J_{\lambda}(\zeta_n(t\Phi)(u_n +t\Phi))\\
& = \left( \frac{1}{p} -\frac{1}{-\alpha +1}\right)\left( a \Vert u_n \Vert^p + \Vert u_n \Vert_{p}^{p}\right)\\
& +b\left( \frac{1}{p\theta} -\frac{1}{-\alpha +1}\right)\Vert u_n \Vert^{p \theta}\\
& +\lambda\left( \frac{q +\alpha -1}{-\alpha +1}\right) \int F(x,u_n^+)\\
& +\left(\frac{1}{-\alpha +1} - \frac{1}{p}\right)\zeta_n^{p}(t\Phi) \left( a \Vert u_n +t\Phi \Vert^p + \Vert u_n +t\Phi \Vert_{p}^{p}\right)\\
& +b\left( \frac{1}{-\alpha +1} -\frac{1}{p\theta}\right)\zeta_n^{p\theta}(t\Phi) \Vert u_n +t\Phi \Vert^{p \theta}\\
& -\lambda \left(  \frac{q +\alpha -1}{-\alpha +1}\right)\zeta^{q}_n(t\Phi) \int F(x,(u_n +t\Phi)^+)dx\\
& =\left(\frac{1}{-\alpha +1} - \frac{1}{p}\right)\left(a (\Vert u_n +t\Phi \Vert^p - \Vert u_n \Vert^p) +\Vert u_n +t\Phi \Vert_{p}^{p} -\Vert u_n \Vert_{p}^{p}\right)\\
& -b\left( \frac{1}{p\theta} -\frac{1}{-\alpha +1}\right)\left(  \Vert u_n +t\Phi \Vert^{p \theta} -\Vert u_n \Vert^{p \theta}\right)\\
& +\left(\frac{1}{-\alpha +1} - \frac{1}{p}\right)(\zeta_n^{p}(t\Phi) -1)\left( a \Vert u_n +t\Phi \Vert^p + \Vert u_n +t\Phi \Vert_{p}^{p}\right)\\
& +b\left( \frac{1}{-\alpha +1} -\frac{1}{p\theta}\right)(\zeta_n^{p\theta}(t\Phi) -1)\Vert u_n +t\Phi \Vert^{p \theta}\\
& -\lambda\left(  \frac{q +\alpha -1}{-\alpha +1}\right)\left( \int F(x,(u_n +t\Phi)^+)dx -\int F(x,u_n^+)dx\right)\\
& -\lambda\left(  \frac{q +\alpha -1}{-\alpha +1}\right)(\zeta^{q}_n(t\Phi) -1)\int F(x,(u_n +t\Phi)^+)dx. \\
\end{aligned}
\end{equation}

deviding by $ t$ and passing to the limmit where $ t \xrightarrow{>} 0,$ we find
{\scriptsize
\begin{equation}
\begin{aligned}
& \langle\zeta^{\prime}_n(0),\Phi \rangle\frac{\Vert u_n\Vert}{n} +\frac{\Vert \Phi\Vert}{n} \geq \left(\frac{p +\alpha -1}{p(-\alpha +1)}\right)\\
&\left(p a\int_{\Omega}\frac{\vert u_n(x) -u_n(y)\vert^{p-2}(u_n(x) -u_n(y))(\Phi(x) -\Phi(y))}{\vert x -y\vert^{n  +ps}}dxdy +\int_{\Omega}(u_n^+)^{p -1}\Phi dx\right)\\
& +\left(\frac{p +\alpha -1}{-\alpha +1}\right)\langle \zeta^{\prime}_n(0),\Phi \rangle\left( a\Vert u_n\Vert^{p} +\Vert u_n\Vert^{p}_{p}\right)\\
& +b\left( \frac{p\theta +\alpha -1}{-\alpha +1}\right)\Vert u_n\Vert^{p(\theta -1)}\int_{\Omega}\frac{\vert u_n(x) -u_n(y)\vert^{p-2}( u_n(x) -u_n(y))(\Phi(x) -\Phi(y))}{\vert x -y\vert^{n  +ps}}dxdy\\
&+b\left( \frac{p\theta +\alpha -1}{-\alpha +1}\right)\langle \zeta^{\prime}_n(0),\Phi \rangle\Vert u\Vert^{p\theta }\\
& -\lambda\left(  \frac{q +\alpha -1}{-\alpha +1}\right)\int_{\Omega} f(x,u_n^+)\Phi dx\\
& -\lambda\left(  \frac{q +\alpha -1}{-\alpha +1}\right) q \langle \zeta^{\prime}_n(0),\Phi \rangle\int F(x,u_n^+)dx. \\
& =\frac{\langle \zeta^{\prime}_n(0),\Phi \rangle}{-\alpha +1}\\
&\left( (p +\alpha -1)\left( a\Vert u_n\Vert^{p} +\Vert u_n\Vert^{p}_{p}\right) +(p\theta +\alpha -1)\Vert u_n\Vert^{p\theta} -\lambda q(q +\alpha -1)\int F(x,u_n^+)dx.\right)\\
& +a\left(\frac{p +\alpha -1}{-\alpha +1}\right)\int_{\Omega}\frac{\vert u_n(x) -u_n(y)\vert^{p-2}(u_n(x) -u_n(y))(\Phi(x) -\Phi(y))}{\vert x -y\vert^{n  +ps}}dxdy
\\
& +\left(\frac{p +\alpha -1}{p(-\alpha +1)}\right)\int_{\Omega}(u_n^+)^{p -1}\Phi dx\\
& +b\left( \frac{p\theta +\alpha -1}{-\alpha +1}\right)\Vert u\Vert^{p(\theta -1)}\int_{\Omega}\frac{\vert u_n(x) -u_n(y)\vert^{p-2}(u_n(x) -u_n(y))(\Phi(x) -\Phi(y))}{\vert x -y\vert^{n  +ps}}dxdy\\
& -\lambda\left(  \frac{q +\alpha -1}{-\alpha +1}\right)\int_{\Omega} f(x,u_n^+)\Phi dx\\
\end{aligned}
\end{equation}
}
which means that
{\scriptsize
\begin{equation}
\begin{aligned}
&\frac{\Vert \Phi\Vert}{n} \geq \frac{\langle\zeta^{\prime}_n(0),\Phi \rangle}{-\alpha +1}\\
&\left( (p +\alpha -1)\left( a\Vert u_n\Vert^{p} +\Vert u_n\Vert^{p}_{p}\right) +(p\theta +\alpha -1)\Vert u\Vert^{p\theta} -\lambda q(q +\alpha -1)\int F(x,u_n^+)dx -\frac{(-\alpha +1)\Vert u_n\Vert}{n} \right)\\
& +a\left(\frac{p +\alpha -1}{-\alpha +1}\right)\int_{\Omega}\frac{\vert u_n(x) -u_n(y)\vert^{p-2}(u_n(x) -u_n(y))(\Phi(x) -\Phi(y))}{\vert x -y\vert^{n  +ps}}dxdy
\\
& +\left(\frac{p +\alpha -1}{p(-\alpha +1)}\right)\int_{\Omega}(u_n^+)^{p -1}\Phi dx\\
& +b\left( \frac{p\theta +\alpha -1}{-\alpha +1}\right)\Vert u_n\Vert^{p(\theta -1)}\int_{\Omega}\frac{\vert u_n(x) -u_n(y)\vert^{p-2}(u_n(x) -u_n(y))(\Phi(x) -\Phi(y))}{\vert x -y\vert^{n  +ps}}dxdy\\
& -\lambda\left(  \frac{q +\alpha -1}{-\alpha +1}\right)\int_{\Omega} f(x,u_n^+)\Phi dx\\
\end{aligned}
\end{equation}
}
By the boundedness of $ u_n$ in $ X_0,$ and lemma \ref{l6} we can choose $ n$ large enough such that
{\scriptsize
\begin{equation}\label{e010}
\begin{aligned}
&(p +\alpha -1)\left( a\Vert u_n\Vert^{p} +\Vert u_n\Vert^{p}_{p}\right) +(p\theta +\alpha -1)\Vert u\Vert^{p\theta} -\lambda q(q +\alpha -1)\int F(x,u_n^+)dx -\frac{(-\alpha +1)\Vert u_n\Vert}{n}\\
& = C_1 - \frac{(-\alpha +1)C_3}{n} >0.
\end{aligned}
\end{equation}
}
That is to say that there exist $ C_4 \in \mathbb{R}$  such that 
\begin{equation*}
\langle\zeta^{\prime}_n(0),\Phi \rangle \leq (-\alpha +1)\left( \frac{\Vert \Phi\Vert}{n} +C_4\right)\left( C_1 - \frac{(-\alpha +1)C_3}{n}\right)^{-1}.
\end{equation*}
Hence by \ref{e010} we get that  $ \langle\zeta^{\prime}_n(0),\Phi \rangle$ is bounded from above for $ \Phi\in X_0$ with $ \Phi \geq 0$ and $ n$ Large enough.
\end{proof}
\begin{lemma}\label{l8}
For $ \lambda < \lambda_{**}$ and $ (u_n)_n \in \mathcal{N}_{\lambda}^{+},$ (resp $ \mathcal{N}_{\lambda}^{-}$) 
satisfyng \ref{ek2} and converges weakly to $ u_*$ in $ X_0$ we have for any $ \Phi \in X_0$ and $ n \in \mathbb{N},$ 
that $ c(x)(u_n^+)^{-\alpha} \in L^{1}(\Omega),$ and moreover we have for any 
$ \Phi \in X_0$ 
\begin{equation}
\begin{aligned}
& \left(a +b\Vert u_n \Vert^{p(\theta -1)}\right) \int_{\mathbb{R}^{2n}} \frac{\vert u_n(x) - u_n(y) \vert^{p-2}(u_n(x) - u_n(y))(\phi(x) - \phi(y))}{\vert x-y \vert^{n + ps}}\\
& + \int_{\Omega}(u_n^+)^{p -1} \phi(x) dx  - \int_{\Omega} c(x) (u_n^+)^{-\alpha} \phi(x) - 
\lambda \int_{\Omega} f(x,u_n^+) \phi(x)dx = o(1).
\end{aligned}
\end{equation}
As $ n \to \infty.$
\end{lemma}
\begin{proof}
Let $ \Phi \in X_0$ with $ \Phi \geq 0,$ then
\begin{equation*}
\begin{aligned}
&(\zeta_n(t\Phi) -1)\frac{\Vert u_n\Vert}{n} +\zeta_n(t\Phi)\frac{\Vert t\Phi\Vert}{n} \geq J_{\lambda}(u_n) -J_{\lambda}(\zeta_n(t\Phi)(u_n +t\Phi))\\
& =-\frac{\zeta^{p}_n(t\Phi) -1}{p}\left(a\Vert u_n\Vert^{p} +\Vert u_n\Vert^{p}_{p}\right)\\
& -\frac{\zeta^{p}_n(t\Phi)}{p}\left(a(\Vert u_n +t\Phi\Vert^{p} -\Vert u_n\Vert^{p}) + \Vert u_n +t\Phi\Vert^{p}_{p} -\Vert u_n\Vert^{p}_{p}\right)\\
& -\frac{\zeta^{p\theta}_n(t\Phi) -1}{p\theta}\Vert u_n\Vert^{p\theta} -b\frac{\zeta^{p\theta}_n(t\Phi)}{p\theta}\left( \Vert u_n -t\Phi\Vert^{p\theta} -\Vert u_n\Vert^{p\theta}\right)\\
& +\frac{\zeta^{-\alpha +1}_n(t\Phi) -1}{-\alpha +1}\int_{\Omega} c(x)((u_n +t\Phi)^+)^{-\alpha +1}\\
&+\frac{1}{-\alpha +1}\int_{\Omega} c(x) 
\left(((u_n -t\Phi)^+)^{-\alpha +1} -(u_n^+)^{-\alpha +1}\right)dx\\
& +\lambda (\zeta^{q}_n(t\Phi) -1)\int_{\Omega}F(x,(u_n +t\Phi)^+) +\lambda\int_{\Omega}(F(x,(u_n +t\Phi)^+) -F(x, u_n^+)dx\\
\end{aligned}
\end{equation*}
Deviding by $ t$ and passing to the limit $ t \xrightarrow{>} 0$ we get 
\begin{equation*}
\begin{aligned}
& \langle\zeta^{\prime}_n(0), \Phi\rangle\frac{\Vert u_n\Vert}{n} +\frac{\Vert\Phi\Vert}{n} \geq\\
& -\langle\zeta^{\prime}_n(0), \Phi\rangle\left(a\Vert u_n\Vert^{p} +\Vert u_n\Vert^{p}_{p}\right)\\
& -a\int_{\Omega}\frac{\vert u_n(x) -u_n(y)\vert^{p -2}(u_n(x) -u_n(y))(\Phi(x) -\Phi(y))}{\vert x -y\vert^{n +ps}}dxdy -\int_{\Omega}(u_n^+)^{p-1}\Phi dx\\
& -\langle\zeta^{\prime}_n(0), \Phi\rangle\Vert u_n\Vert^{p\theta}\\
& -b\Vert u_n\Vert^{p(\theta -1)}\int_{\Omega}\frac{\vert u_n(x) -u_n(y)\vert^{p -2}(u_n(x) -u_n(y))(\Phi(x) -\Phi(y))}{\vert x -y\vert^{n +ps}}dxdy\\
& +\langle\zeta^{\prime}_n(0), \Phi\rangle\int_{\Omega}c(x)(u_n^+)^{-\alpha +1}dx\\\
& +\frac{1}{-\alpha +1}\liminf_{t \xrightarrow{>} 0}
\frac{\int_{\Omega} c(x) \left(((u_n -t\Phi)^+)^{-\alpha +1} -(u_n^+)^{-\alpha +1}\right)}{t}dx\\
& +\lambda q\langle\zeta^{\prime}_n(0), \Phi\rangle\int_{\Omega} F(x,u_n^+)dx +\lambda \int_{\Omega}f(x, u_n^+)\Phi dx\\
& = -\langle\zeta^{\prime}_n(0), \Phi\rangle\\
&\left( \left(a\Vert u_n\Vert^{p} +\Vert u_n\Vert^{p}_{p}\right) +\Vert u_n\Vert^{p\theta} +\int_{\Omega}c(x) ( u_n^+)^{-\alpha +1}dx +\lambda q \int_{\Omega} F(x,u_n^+)dx\right)\\
& -\left( a +b\Vert u_n\Vert^{p(\theta -1)}\right)\int_{\Omega}\frac{\vert u_n(x) -u_n(y)\vert^{p -1}(\Phi(x) -\Phi(y))}{\vert x -y\vert^{n +ps}}dxdy -\int_{\Omega}\vert u_n\vert^{p-1}\Phi dx\\
&+\frac{1}{-\alpha +1}\liminf_{t \xrightarrow{>} 0}\int_{\Omega} \frac{c(x) \left(\vert u_n -t\Phi\vert^{-\alpha +1} -\vert u_n\vert^{-\alpha +1}\right)}{t}dx ++\lambda \int_{\Omega}f(x, u_n)\Phi dx\\
& =-\left( a +b\Vert u_n\Vert^{p(\theta -1)}\right)\int_{\Omega}\frac{\vert u_n(x) -u_n(y)\vert^{p -2}(u_n(x) -u_n(y))(\Phi(x) -\Phi(y))}{\vert x -y\vert^{n +ps}}dxdy \\
&-\int_{\Omega}(u_n^+)^{p-1}\Phi dx +\frac{1}{-\alpha +1}\liminf_{t \xrightarrow{>} 0}\int_{\Omega} \frac{c(x) \left(((u_n -t\Phi)^+)^{-\alpha +1} -( u_n^+)^{-\alpha +1}\right)}{t}dx\\
& +\lambda \int_{\Omega}f(x, u_n^+)\Phi dx\\
\end{aligned}
\end{equation*}
Hence, by the boundedness of $ u_n$ in $ X_0$ and lemma \ref{l7} and \ref{e011}, we get 
\begin{equation*}
\liminf_{t \xrightarrow{>} 0}\int_{\Omega} \frac{c(x) \left(((u_n -t\Phi)^+)^{-\alpha +1} -(u_n^+)^{-\alpha +1}\right)}{t}dx <\infty.
\end{equation*}
Also Since $ c(x) \left(((u_n -t\Phi)^+)^{-\alpha +1} -(u_n^+)^{-\alpha +1}\right) \geq 0$ then by Fatous lemma and lemma \ref{l7} we have 
\begin{equation*}
\begin{aligned}
&\int_{\Omega} \liminf_{t \to 0^+}\frac{c(x) \left(((u_n -t\Phi)^+)^{-\alpha +1} -(u_n^+)^{-\alpha +1}\right)}{t}dx = \int c(x)(u_n^+)^{-\alpha}\Phi \\
&\leq \liminf_{t \to 0^+}\int_{\Omega} \frac{c(x) \left(((u_n -t\Phi)^+)^{-\alpha +1} -(u_n^+)^{-\alpha +1}\right)}{t}dx 
\end{aligned}
\end{equation*}
and there exist canstant $ C_5$ given by the boundedness of $ u_n$ and of $ \langle\zeta^{\prime}_n(0),\Phi \rangle$ such that
{\tiny
\begin{equation*}
\begin{aligned}
&\frac{1}{-\alpha +1}\int_{\Omega}c(x)(u_n^+)^{-\alpha}\Phi dx \leq \liminf_{t \to 0^+}\frac{1}{-\alpha +1}\int_{\Omega}\frac{c(x)\left(((u_n(x) +t\Phi)^+)^{-\alpha +1} -(u_n^+)^{-\alpha +1}\right)}{t}dx\\
& \leq \frac{\langle\zeta^{\prime}_n(0), \Phi\rangle\Vert u_n\Vert +\Vert\Phi\Vert}{n} +\left( a +b\Vert u_n\Vert^{p(\theta -1)}\right)\int_{\Omega}\frac{\vert u_n(x) -u_n(y)\vert^{p -2}(u_n(x) -u_n(y))(\Phi(x) -\Phi(y))}{\vert x -y\vert^{n +ps}}dxdy\\
& +\int_{\Omega}(u_n^+)^{p-1}\Phi dx -\lambda \int_{\Omega}f(x, u_n^+)\Phi dx\\
& \leq \frac{C_5 +\Vert \Phi\Vert}{n} +\left( a +b\Vert u_n\Vert^{p(\theta -1)}\right)\int_{\Omega}\frac{\vert u_n(x) -u_n(y)\vert^{p -2}(u_n(x) -u_n(y))(\Phi(x) -\Phi(y))}{\vert x -y\vert^{n +ps}}dxdy\\
& +\int_{\Omega}(u_n^+)^{p-1}\Phi dx -\lambda \int_{\Omega}f(x, u_n^+)\Phi dx\\
\end{aligned}
\end{equation*}
}
This and \ref{e02} yields that $ c(x)\vert u_n\vert^{-\alpha}$ is $ L^{1}-$integrable function. Also as $ n \to \infty.$ we get
\begin{equation*}
\begin{aligned}
&\left( a +b\Vert u_n\Vert^{p(\theta -1)}\right)\int_{\Omega}\frac{\vert u_n(x) -u_n(y)\vert^{p -2}(u_n(x) -u_n(y))(\Phi(x) -\Phi(y))}{\vert x -y\vert^{n +ps}}dxdy \\
& +\int_{\Omega}(u_n^+)^{p-1}\Phi dx-\int_{\Omega}c(x)(u_n^+)^{-\alpha}\Phi dx -\lambda \int_{\Omega}f(x, u_n^+)\Phi dx \geq o_n(1).
\end{aligned}
\end{equation*}
The last holds for every $ \Phi \in X_0$ for this let $ \Psi$ a test function such that $ \Psi^+ =(u_n^+ +\epsilon\Phi)^{+}, $ hence
\begin{equation}\label{e016}
\begin{aligned}
&o_n(1) \leq \left( a +b\Vert u_n\Vert^{p(\theta -1)}\right)\\
&\int_{\Omega}\frac{\vert u_n(x) -u_n(y)\vert^{p -2}(u_n(x) -u_n(y))(\Psi^{+}(x) -\Psi^{+}(y))}{\vert x -y\vert^{n +ps}}dxdy \\
&+\int_{\Omega}(u_n^+)^{p-1}\Psi^{+} dx -\int_{\Omega}c(x)(u_n^+)^{-\alpha}\Psi^{+} dx -\lambda \int_{\Omega}f(x, u_n^+)\Psi^{+} dx \\
&= \left( a +b\Vert u_n\Vert^{p(\theta -1)}\right)\\
&\int_{\Omega}\frac{\vert u_n(x) -u_n(y)\vert^{p -2}(u_n(x) -u_n(y))((\Psi +\Psi^{-})(x) -(\Psi +\Psi^{-})(y))}{\vert x -y\vert^{n +ps}}dxdy \\
&+\int_{\Omega}(u_n^+)^{p-1}(\Psi +\Psi^{-})(x) dx -\int_{\Omega}c(x)(u_n^+)^{-\alpha}(\Psi +\Psi^{-})(x) dx\\
& -\lambda \int_{\Omega}f(x, u_n^+)(\Psi +\Psi^{-})(x)dx \\
\end{aligned}
\end{equation}
And as in \cite{fiscella2019nehari} we have the following inequalities,
\begin{equation}\label{e013}
\begin{aligned}
&(u(x) -u(y))(u^-(x) -u^-(y)) \leq - \vert u^-(x) -u^-(y)\vert^{2}\\
&(u(x) -u(y))(u^+(x) -u^+(y)) \leq \vert u(x) -u(y)\vert^{2}
\end{aligned}
\end{equation}
therefore we get the following 
{\scriptsize
\begin{equation}\label{e014}
\begin{aligned}
&\int_{\Omega}\vert u_n\vert^{p-1}\Psi^{+} dx 
\leq \int_{\Omega}\vert u_n\vert^{p}dx +\epsilon\int_{\Omega}\vert u_n\vert^{p-1}\Phi dx -\int_{\Omega}\vert u_n\vert^{p-1}\Psi^{-}dx.\\
& \textit{and}\\
&\int_{\Omega}c(x)\vert u_n\vert^{-\alpha}(\Psi -\Psi^{-})(x) dx \leq \int_{\Omega}c(x)\vert u_n\vert^{-\alpha +1} +\epsilon\int_{\Omega}c(x)\vert u_n\vert^{-\alpha}\Phi(x) dx -\int_{\Omega}c(x)\vert u_n\vert^{-\alpha}\Psi^{-}(x)dx.\\
& \textit{and}\\
&\int_{\Omega}f(x, u_n)(\Psi -\Psi^{-})(x)dx \leq \int_{\Omega}f(x, u_n)u_n dx + \epsilon\int_{\Omega}f(x, u_n)\Phi(x)dx -\int_{\Omega}f(x, u_n)\Psi^{-}(x)dx.
\end{aligned}
\end{equation}
}
using \ref{e013} and \ref{e014} we find
\begin{equation}\label{e015}
\begin{aligned}
&\int_{\mathbb{R}^{2n}}\frac{\vert u_n(x) -u_n(y)\vert^{p -2}(u_n(x) -u_n(y))((\Psi +\Psi^{-})(x) -(\Psi +\Psi^{-})(y))}{\vert x -y\vert^{n +ps}}dxdy\\
&\leq \int_{\mathbb{R}^{2n}}\frac{\vert u_n(x) -u_n(y)\vert^{p}}{\vert x -y\vert^{n +ps}}dxdy \\
&+\epsilon\int_{\mathbb{R}^{2n}}\frac{\vert u_n(x) -u_n(y)\vert^{p -2}(u_n(x) -u_n(y))(\Phi(x) -\Phi(y))}{\vert x -y\vert^{n +ps}}dxdy\\
&+\int_{\mathbb{R}^{2n}}\frac{\vert u_n(x) -u_n(y)\vert^{p -2}(u_n(x) -u_n(y))(\Psi^{-}(x) -\Psi^{-}(y))}{\vert x -y\vert^{n +ps}}dxdy
\end{aligned}
\end{equation}
Hence, by \ref{e015} and \ref{e016} as $ n \to \infty$ we get
{\footnotesize
\begin{equation}
\begin{aligned}
& o_n(1) \leq \left( a +b\Vert u_n\Vert^{p(\theta -1)}\right)\Vert u_n\Vert^{p}\\
& +\epsilon\left( a +b\Vert u_n\Vert^{p(\theta -1)}\right)\int_{\mathbb{R}^{2n}}\frac{\vert u_n(x) -u_n(y)\vert^{p -2}(u_n(x) -u_n(y))(\Phi(x) -\Phi(y))}{\vert x -y\vert^{n +ps}}dxdy\\
& +\left( a +b\Vert u_n\Vert^{p(\theta -1)}\right)\int_{\mathbb{R}^{2n}}\frac{\vert u_n(x) -u_n(y)\vert^{p -2}(u_n(x) -u_n(y))(\Psi^{-}(x) -\Psi^{-}(y))}{\vert x -y\vert^{n +ps}}dxdy\\
& +\Vert u_n\Vert^{p}_{p} +\epsilon\int_{\Omega}(u_n^+){p-1}\Phi dx +\int_{\Omega}(u_n^+)^{p-1}\Psi^{-}dx\\
& -\int_{\Omega}c(x)(u_n^+)^{-\alpha}(\Psi +\Psi^{-})(x) dx -\lambda \int_{\Omega}f(x, u_n^+)(\Psi +\Psi^{-})(x)dx\\
& =\left( a +b\Vert u_n\Vert^{p(\theta -1)}\right)\Vert u_n\Vert^{p} +\Vert u_n\Vert^{p}_{p} -\int_{\Omega}c(x)(u_n^+)^{-\alpha +1}dx -\lambda \int_{\Omega}f(x, u_n^+)u_n^+ dx\\
& +\epsilon \left( \left( a +b\Vert u_n\Vert^{p(\theta -1)}\right)\int_{\mathbb{R}^{2n}}\frac{\vert u_n(x) -u_n(y)\vert^{p -2}(u_n(x) -u_n(y))(\Phi(x) -\Phi(y))}{\vert x -y\vert^{n +ps}}dxdy  \right)\\
&+ \epsilon\left(\int_{\Omega}(u_n^+)^{p-1}\Phi dx -\int_{\Omega}c(x)( u_n^+)^{-\alpha}\Phi(x) dx -\lambda \int_{\Omega}f(x, u_n^+)\Phi(x)dx \right)\\
& +\left( a +b\Vert u_n\Vert^{p(\theta -1)}\right)\int_{\mathbb{R}^{2n}}\frac{\vert u_n(x) -u_n(y)\vert^{p -2}(u_n(x) -u_n(y))(\Psi^{-}(x) -\Psi^{-}(y))}{\vert x -y\vert^{n +ps}}dxdy\\
& +\int_{\Omega}(u_n^+)^{p-1}\Psi^{-}dx +\int_{\{u_n^+ +t\Phi \leq 0\}}c(x)(u_n^+)^{-\alpha}(u_n^+ +t\Phi)(x)dx\\
&+\lambda\int_{\{u_n^+ +t\Phi \leq 0\}}f(x, u_n^+)(u_n^+ +t\Phi)(x)dx.
\end{aligned}
\end{equation}
}
we know that $ u_n \in \mathcal{N}_{\lambda}$ and $ c(x) >0,$ then as $ n \to \infty.$ we obtain
{\footnotesize
\begin{equation}\label{e020}
\begin{aligned}
& =\left( a +b\Vert u_n\Vert^{p(\theta -1)}\right)\Vert u_n\Vert^{p} +\Vert u_n\Vert^{p}_{p} -\int_{\Omega}c(x)(u_n^+)^{-\alpha +1}dx -\lambda \int_{\Omega}f(x, u_n^+)u_n^+ dx\\
& +\epsilon \left( \left( a +b\Vert u_n\Vert^{p(\theta -1)}\right)\int_{\mathbb{R}^{2n}}\frac{\vert u_n(x) -u_n(y)\vert^{p -2}(u_n(x) -u_n(y))(\Phi(x) -\Phi(y))}{\vert x -y\vert^{n +ps}}dxdy \right)\\
&+ \epsilon\left(\int_{\Omega}(u_n^+)^{p-1}\Phi dx -\int_{\Omega}c(x)(u_n^+)^{-\alpha}\Phi(x) dx -\lambda \int_{\Omega}f(x, u_n^+)\Phi(x)dx \right)\\
& +\left( a +b\Vert u_n\Vert^{p(\theta -1)}\right)\int_{\mathbb{R}^{2n}}\frac{\vert u_n(x) -u_n(y)\vert^{p -2}(u_n(x) -u_n(y))(\Psi^{-}(x) -\Psi^{-}(y))}{\vert x -y\vert^{n +ps}}dxdy \\
& +\int_{\Omega}(u_n^+)^{p-1}\Psi^{-}dx +\lambda\int_{\{u_n^+ +\epsilon\Phi \leq 0\}}f(x, u_n^+)(u_n^+ +\epsilon\Phi)(x)dx.
\end{aligned}
\end{equation}
}
On the other hand we have
{\tiny
\begin{equation}\label{e40}
\begin{aligned}
&\mathbb{R}^{2n} \\
&=\left( (\mathbb{R}^n \setminus \{x \in \mathbb{R}^n : \Psi(x) \leq 0\})\cup\{x \in \mathbb{R}^n : \Psi(x) \leq 0\} \right)\times\left(  (\mathbb{R}^n \setminus \{x \in \mathbb{R}^n : \Psi(x) \leq 0\})\cup\{x \in \mathbb{R}^n : \Psi(x) \leq 0\}\right)\\
&=\{x \in \mathbb{R}^n : \Psi(x) \leq 0\}\times \{x \in \mathbb{R}^n : \Psi(x) \leq 0\}\cup(\mathbb{R}^n \setminus \{x \in \mathbb{R}^n : \Psi(x) \leq 0\})\times\{x \in \mathbb{R}^n : \Psi(x) \leq 0\}\cup \\
&\{x \in \mathbb{R}^n : \Psi(x) \leq 0\}\times(\mathbb{R}^n \setminus \{x \in \mathbb{R}^n : \Psi(x) \leq 0\})\cup(\mathbb{R}^n \setminus \{x \in \mathbb{R}^n \Psi(x) \leq 0\})\times (\mathbb{R}^n \setminus \{x \in \mathbb{R}^n : \Psi(x) \leq 0\})
\end{aligned}
\end{equation}
}
Then by \ref{e40} and \ref{e013} we find
\begin{equation}
\begin{aligned}
&\int_{\mathbb{R}^{2n}}\frac{\vert u_n(x) -u_n(y)\vert^{p -2}(u_n(x) -u_n(y))(\Psi^{-}(x) -\Psi^{-}(y))}{\vert x -y\vert^{n +ps}}dxdy \\
& \leq 2\epsilon\int_{\{ x \in \mathbb{R}^n : \Psi(x) \leq 0\}\times \mathbb{R}^n}\frac{\vert u_n(x) -u_n(y)\vert^{p -1}\vert\Phi(x) -\Phi(y)\vert}{\vert x -y\vert^{n +ps}}dxdy \\
\end{aligned}
\end{equation}
by Holder inequality we obtain 
{\tiny
\begin{equation}
\int_{\{ x \in \mathbb{R}^n : \Psi(x) \leq 0\}\times \mathbb{R}^n}\frac{\vert u(x) -u(y)\vert^{p -1}\vert \Phi(x) -\Phi
(y)\vert}{\vert u(x) -u(y)\vert^{p\frac{n +ps}{p}}}dxdy\leq C_6 \left( \int_{\{ x \in \mathbb{R}^n : \Psi(x) \leq 0\}\times \mathbb{R}^n}\frac{\vert \Phi(x) -\Phi(y)\vert^{p}}{\vert x -y\vert^{n +ps}}\right)^{\frac{1}{p}}
\end{equation}
}
and we can see that $ \frac{\vert\Phi(x) -\Phi(y)\vert}{\vert x -y\vert^{\frac{n +ps}{p}}} \in L^{p}(\mathbb{R}^{2n})$ and $ \Phi$ has compact supp So there exist $ r >0$ large enough such that $ supp(\Phi) \subset B_{r}.$ So $ \vert \{ x \in \mathbb{R}^n : \Psi(x) \leq 0\} \times B_{r}\vert \longrightarrow 0$ as $ \epsilon \xrightarrow{>} 0$
therefore
\begin{equation}\label{e017}
\lim_{\epsilon \xrightarrow{>} 0}\int_{(\{ x \in \mathbb{R}^n : u^+(x) +\epsilon \Phi(x) \leq 0\})\times(\mathbb{R}^n  \setminus B_{r})}\frac{\vert \Phi(x) -\Phi(y)\vert^{p}}{\vert x -y\vert^{n +ps}}dxdy =0,
\end{equation}
and we have by Holder inequality
\begin{equation*}
\begin{aligned}
&\left\vert \int_{\{u_n^+ +\epsilon\Phi \leq 0\}}f(x, u_n)(u_n^+ +\epsilon\Phi)(x)dx\right\vert \\
&\leq \gamma \Vert u_n\Vert_{q}^{q}\int_{\{u_n^+ +\epsilon\Phi \leq 0\}}1dx +C_{7}\epsilon\int_{\{u_n^+ +\epsilon\Phi \leq 0\}}\vert\Phi(x)\vert^{p_s^*}dx
\end{aligned}
\end{equation*}
Since $ \{u_n^+ +\epsilon\Phi \leq 0\} \longrightarrow 0$ as $ \epsilon \xrightarrow{>} 0,$ we can conclude that
\begin{equation}\label{e018}
\begin{aligned}
&\lim_{\epsilon  \to 0}\frac{1}{\epsilon}\int_{\{u_n^+ +t\Phi \leq 0\}}f(x, u_n)(u_n^+ +\epsilon\Phi)(x)dx =0.
\end{aligned}
\end{equation}
Also we have 
\begin{equation*}
\int_{\Omega}(u_n^+)^{p-1}\Psi^{-}dx  \leq \epsilon \int_{\{u_n^+ +\epsilon\Phi \leq 0\}}(u_n^+)^{p-1}\vert\Phi(x)\vert dx
\end{equation*}
Since $ \{u_n^+ +\epsilon\Phi \leq 0\} \longrightarrow 0$ as $ \epsilon \xrightarrow{>} 0,$ we can conclude that 
\begin{equation}\label{e019}
\lim_{\epsilon \to 0}\frac{1}{\epsilon}\int_{\Omega}(u_n^+)^{p-1}\Psi^{-}dx =0.
\end{equation}
Using \ref{e017}, \ref{e018}, \ref{e019} in \ref{e020} we find
\footnotesize
\begin{equation*}
\begin{aligned}
& =\left( a +b\Vert u_n\Vert^{p(\theta -1)}\right)\Vert u_n\Vert^{p} +\Vert u_n\Vert^{p}_{p} -\int_{\Omega}c(x)(u_n^+)^{-\alpha +1}dx -\lambda \int_{\Omega}f(x, u_n^+)u_n^+ dx\\
& +\epsilon \left( \left( a +b\Vert u_n\Vert^{p(\theta -1)}\right)\int_{\mathbb{R}^{2n}}\frac{\vert u_n(x) -u_n(y)\vert^{p -2}(u_n(x) -u_n(y))(\Phi(x) -\Phi(y))}{\vert x -y\vert^{n +ps}}dxdy \right)\\
&+ \epsilon\left(\int_{\Omega}(u_n^+)^{p-1}\Phi dx -\int_{\Omega}c(x)(u_n^+)^{-\alpha}\Phi(x) dx -\lambda \int_{\Omega}f(x, u_n^+)\Phi(x)dx \right)\\
& +\left( a +b\Vert u_n\Vert^{p(\theta -1)}\right)\int_{\mathbb{R}^{2n}}\frac{\vert u_n(x) -u_n(y)\vert^{p -2}(u_n(x) -u_n(y))(\Psi^{-}(x) -\Psi^{-}(y))}{\vert x -y\vert^{n +ps}}dxdy \\
& +\int_{\Omega}(u_n^+)^{p-1}\Psi^{-}dx +\lambda\int_{\{u_n^+ +\epsilon\Phi \leq 0\}}f(x, u_n^+)(u_n^+ +\epsilon\Phi)(x)dx.
\end{aligned}
\end{equation*}
as $ \epsilon \xrightarrow{>} 0$ for all $ \Phi \in X_0$
\begin{equation}
\begin{split}
o_n(1) \leq \left( a +b\Vert u_n\Vert^{p(\theta -1)}\right)\int_{\mathbb{R}^{2n}}\frac{\vert u_n(x) -u_n(y)\vert^{p -2}(u_n(x) -u_n(y))(\Phi(x) -\Phi(y))}{\vert x -y\vert^{n +ps}}dxdy \\
+\int_{\Omega}(u_n^+)^{p-1}\Phi dx -\int_{\Omega}c(x)(u_n^+)^{-\alpha}\Phi(x) dx -\lambda \int_{\Omega}f(x, u_n^+)\Phi(x)dx 
\end{split}
\end{equation}
\end{proof}
\section{Compactness properties}
\begin{proposition}\label{pr1}
Let $\lambda \in (0, \lambda_*),$ and $ (u_n)_n \in \mathcal{N}_{\lambda}$ where $ u_n \rightharpoonup u_*$ weakly in $ X_0$ and satifing \ref{ek2} and
$ J_{\lambda}(u_n) \longrightarrow c < c_{\lambda}$ as $ n \to \infty.$ with 
{\scriptsize
\begin{equation}\label{lev}
\begin{aligned}
& c_{\lambda} =\frac{s}{n} S_{p}^{\frac{n}{s p}}a^{\frac{n}{sp}} \gamma^{-\frac{n}{sp_s^*}}(p_s^* \lambda)^{-\frac{n}{s p_s^*}} -\left(\frac{p_s^* +\alpha -1}{p_s^*}\Vert u\Vert_{\infty}\vert \Omega\vert^{\frac{p_s^* +\alpha -1}{p_s^*}}S_{p}^{\frac{\alpha -1}{p}}\right)^{\frac{p\theta}{p\theta +\alpha -1}}\\
& \left( b\frac{p_s^* -p\theta}{p_s^*p\theta}\right)^{\frac{\alpha -1}{p\theta +\alpha -1}}\left( \frac{\alpha}{1 -\alpha}\right)
 \end{aligned}
\end{equation}
}
Then, the sequence $ (u_n)_n$ has a strongly convergent subseqence to $ u_*$ in $ X_0.$ 
\end{proposition}
\begin{proof}
Let $ (u_n)\in \mathcal{N}_{\lambda},$ where $ u_n \rightharpoonup u_*,$ in $ X_0$ hence $ u_n$ and $ u_n^-$ is bounded in $ X_0.$ then by lemma \ref{l8}
\begin{equation*}
 (a +b\Vert u_n \Vert^{p(\theta -1)}) \int_{\mathbb{R}^{2n}} \frac{\vert u_n(x) - u_n(y) \vert^{p-2}(u_n(x) - u_n(y))(u_n^-(x) - u_n^-(y))}{\vert x-y \vert^{n + ps}} \longrightarrow 0.
\end{equation*}
By \ref{e013} we obtain
\begin{equation*}
-\lim_{n \to \infty}(a +b\Vert u_n \Vert^{p(\theta -1)}) \int_{\mathbb{R}^{2n}}\frac{\vert u_n^-(x) - u_n^-(y) \vert^{p}}{\vert x-y \vert^{n + ps}} \geq 0.
\end{equation*}
which means that $ \Vert u_n^-\Vert \longrightarrow 0,$ as $ n \to \infty.$ Therefore we sppose that $ u_n$ is sequence of positive functions,  hence by \cite{servadei2012mountain} and \cite{brezis2011functional} theorem 4.9, there exist subsequence denoted $ (u_n)_n$ where
\begin{equation}\label{e021}
\begin{aligned}
& u_n \rightharpoonup u_* \quad \textit{in}   \quad L^{p_s^*}(\Omega),?? X_0,\\
& u_n \longrightarrow u_* \quad \textit{in}   \quad L^{q}(\Omega) \,\, \textit{for any}\,\, q \in [1,p_s^*),\\
& u_n \longrightarrow u_* \quad \textit{a.e.} \quad \Omega, 
\end{aligned}
\end{equation}
In particular there exist $ h\in L^{q}(\Omega),$ such that $ u_n(x) \leq h(x)$ for any $ n \in \mathbb{N}$ and for a.e. $ x \in \Omega.$ we specify 
that \ref{e021} are satified even in $ \mathbb{R}^n,$ by taking $ u_n =0$ and $ h(x) =0$ a.e. in $ \mathbb{R}^n \setminus \Omega.$ Moreover by \ref{e021}
there exist $ \mu \geq 0,$ such that $ \Vert u_n\Vert \longrightarrow \mu,$ assuming that $ \mu >0,$ hence by \ref{e021}, \cite{brezis1983relation}
and by Brezis-lieb lemma [\cite{brezislieb} theorem 2] we have as $ n \to \infty.$
\begin{equation}
\begin{aligned}
 \Vert u_n\Vert^{p} &=\Vert u_n -u_*\Vert^{p} +\Vert u_*\Vert^{p} +o(1),\\
 \int_{\Omega}F(x,u_n)dx &=\int_{\Omega}F(x,u_n -u_*)dx +\int_{\Omega}F(x,u_*) +o(1).
\end{aligned}
\end{equation}
Hence from lemma \ref{l8} and \ref{Euler} we get that with $ \Phi = u_n -u_*$ and by \ref{e021} we get that $ \int_{\Omega}  u_n^{p-1}(u_n -u_*)(x) dx \longrightarrow 0.$ Then
\begin{equation*}
\begin{aligned}
o(1) &=(a +b\Vert u_n \Vert^{p(\theta -1)})\\
&\int_{\mathbb{R}^{2n}} \frac{\vert u_n(x) - u_n(y) \vert^{p-2}(u_n(x) - u_n(y))((u_n -u_*)(x) - (u_n -u_*))(y))}{\vert x-y \vert^{n + ps}}\\
& +\int_{\Omega}  u_n^{p-1}(u_n -u_*)(x) dx -\int_{\Omega} c(x)u_n^{-\alpha} (u_n -u_*)(x)dx\\
& -\lambda \int_{\Omega} f(x,u_n) (u_n -u_*)(x)dx\\
& =(a +b\mu^{p(\theta -1)})(\mu^{p} -\Vert u_*\Vert^{p}) -\int_{\Omega} c(x)u_n^{-\alpha} (u_n -u_*)(x)dx\\
&-\lambda \int_{\Omega} f(x,u_n)u_n dx +\lambda \int_{\Omega} f(x,u_n)u_*dx\\
& =(a +b\mu^{p(\theta -1)})(\mu^{p} -\Vert u_*\Vert^{p}) -\int_{\Omega} c(x)u_n^{-\alpha} (u_n -u_*)(x)dx\\
& -\lambda p_s^*\int_{\Omega} F(x,u_n) dx +\lambda \int_{\Omega} f(x,u_*)u_*dx\\
& =(a +b\mu^{p(\theta -1)})(\mu^{p} -\Vert u_*\Vert^{p}) -\int_{\Omega} c(x)u_n^{-\alpha} (u_n -u_*)(x)dx \\
&-\lambda p_s^*\left(\int_{\Omega} F(x,u_n) -F(x,u_*)dx\right)\\
& =(a +b\mu^{p(\theta -1)})(\Vert u_n -u_*\Vert^{p}) -\int_{\Omega} c(x)u_n^{-\alpha} (u_n -u_*)(x)dx \\
&-\lambda p_s^* \int_{\Omega} F(x,u_n -u_*)dx +o(1)\\
\end{aligned}
\end{equation*}
Then 
{\footnotesize
\begin{equation}\label{e030}
\lim_{n \to \infty}(a +b\mu^{p(\theta -1)})\Vert u_n -u_*\Vert^{p} =\lim_{n \to \infty}\left( \int_{\Omega} c(x)u_n^{-\alpha} (u_n -u_*)(x)dx +\lambda p_s^* \int_{\Omega} F(x,u_n -u_*)dx\right)
\end{equation}
}
Put 
\begin{equation}
\lim_{n \to \infty} \int_{\Omega} F(x,u_n -u_*)dx =d.
\end{equation}
By \ref{e021}, \ref{e02} and the the dominated comvergence theorem we have
\begin{equation}
\lim_{n \to \infty}\int_{\Omega} c(x)u_n^{-\alpha +1}dx =\int_{\Omega} c(x)u_*^{-\alpha +1}dx.
\end{equation}
Moreover for all $ n \in \mathbb{N}$ we have $ c(x)u_n^{-\alpha +1} \in L^{1}(\Omega),$ hence by Fatous lemma we obtain
\begin{equation}
\int_{\Omega} c(x)u_*^{-\alpha +1}dx \leq \lim_{n \to \infty}\inf\int_{\Omega} c(x)u_n^{-\alpha +1}dx.
\end{equation}
hence by \ref{e030} we obtain
\begin{equation}
\lim_{n \to \infty}(a +b\mu^{p(\theta -1)})\Vert u_n -u_*\Vert^{p} \leq \lambda p_s^* d
\end{equation}
If $ d =0,$ the proof is complete, assume the opposite that is $ d >0,$ then by the Sobolev fractional constant \ref{e3} and \ref{e1} we obtain
\begin{equation}
S_{p}d^{\frac{p}{p_s^*}}\gamma^{-\frac{p}{p_s^*}} \leq \Vert u_n -u_*\Vert^{p}.
\end{equation}
then we have as $ n \to \infty$ 
\begin{equation}
(\lambda^{-1} p_s^*) S_{p} \gamma^{-\frac{p}{p_s^*}}(a +b\mu^{p(\theta -1)}) \leq d^{\frac{p_s^* -p}{p^s_*}}.
\end{equation}
Also
\begin{equation}
(a +b\mu^{p(\theta -1)})(\mu^{p} -\Vert u_*\Vert^{p}) \leq \lambda p_s^* d
\end{equation}
then we get 
\begin{equation*}
\begin{aligned}
 & d \geq (\lambda p_s^*)^{-1} (a +b\mu^{p(\theta -1)})(\mu^{p} -\Vert u_*\Vert^{p})\\
 & d^{\frac{p_s^* -p}{p}} \geq (\lambda p_s^*)^{-\frac{p_s^* -p}{p}} (a +b\mu^{p(\theta -1)})^{\frac{p_s^* -p}{p}}(\mu^{p} -\Vert u_*\Vert^{p})^{\frac{p_s^* -p}{p}}\\
 & =(\lambda p_s^*)^{-\frac{p_s^* -p}{p}} (a +b\mu^{p(\theta -1)})^{\frac{p_s^* -p}{p}}\left( \lim_{n \to \infty }\Vert u_n -u_*\Vert^{p}\right)^{\frac{p_s^* -p}{p}}\\
 & \geq (\lambda p_s^*)^{-\frac{p_s^* -p}{p}} (a +b\mu^{p(\theta -1)})^{\frac{p_s^* -p}{p}}\left( S_{p}d^{\frac{p}{p_s^*}}\gamma^{-\frac{p}{p_s^*}}\right)^{\frac{p_s^* -p}{p}}\\
 & =(\lambda p_s^*)^{-\frac{p_s^* -p}{p}}(a +b\mu^{p(\theta -1)})^{\frac{p_s^* -p}{p}}S_{p}^{\frac{s p_s^*}{n}} d^{\frac{p_s^* -p}{p^*_s}}\gamma^{-\frac{p_s^* -p}{p^*_s}}\\
 & \geq (\lambda p_s^*)^{-\frac{p_s^* -p}{p}}(a +b\mu^{p(\theta -1)})^{\frac{p_s^* -p}{p}}S_{p}^{\frac{s p_s^*}{n}}\gamma^{-\frac{p_s^* -p}{p^*_s}}\left( \lambda^{-1} S_{p} \gamma^{-\frac{p}{p_s^*}}(a +b\mu^{p(\theta -1)})\right)\\
 & = (\lambda p_s^*)^{-\frac{p_s^*}{p}}(a +b\mu^{p(\theta -1)})^{\frac{p_s^*}{p}}S_{p}^{\frac{p_s^*}{p}}\gamma^{-1}
 \end{aligned}
\end{equation*}
then we get
\begin{equation*}
\left(\mu^{p}\right)^{\frac{p_s^* -p}{p}} \geq (\mu^{p} -\Vert u_*\Vert^{p})^{\frac{p_s^* -p}{p}} \geq (\lambda p_s^*)^{-1}(a +b\mu^{p(\theta -1)})S_{p}^{\frac{p_s^* }{p}}\gamma^{-1}
\end{equation*}
that is to say
\begin{equation}\label{e031}
\mu^{p} \geq \left(\frac{a +b\mu^{p(\theta -1)}}{\gamma \lambda p_s^*}\right)^{\frac{p}{p_s^* -p}}S_{p}^{\frac{n}{s p}}
\end{equation}
Then we get from \ref{j} and \ref{el1} for $ n \in \mathbb{N}$ we get
\begin{equation*}
\begin{aligned}
J_{\lambda}(u_n) -\frac{1}{p_s^*}\langle J^{\prime}_{\lambda}(u_n), u_n\rangle & =\left( \frac{1}{p} -\frac{1}{p_s^*}\right)\left(a \Vert u_n\Vert^p + \Vert u_n \Vert^{p}_{p}\right) +b\left( \frac{1}{p\theta} -\frac{1}{p_s^*}\right)\Vert u_n \Vert^{p \theta} \\
& -\left( \frac{1}{-\alpha +1} -\frac{1}{p_s^*}\right)\int c(x) u_n^{-\alpha+1}dx\\
&\geq \left( \frac{1}{p} -\frac{1}{p_s^*}\right)\left(a \Vert u_n\Vert^p + \Vert u_n \Vert^{p}_{p}\right) +b\left( \frac{1}{p\theta} -\frac{1}{p_s^*}\right)\Vert u_n \Vert^{p \theta} \\
& -\left( \frac{1}{-\alpha +1} -\frac{1}{p_s^*}\right)\Vert c \Vert_{L^{\infty}}\vert \Omega\vert^{\frac{p_s^* +\alpha -1}{p_s^*}}S_{p}^{\frac{\alpha-1}{p}} \Vert u \Vert^{-\alpha+1}\\
&\geq a\left( \frac{1}{p} -\frac{1}{p_s^*}\right)\Vert u_n\Vert^p +b\left( \frac{1}{p\theta} -\frac{1}{p_s^*}\right)\Vert u_n \Vert^{p \theta} \\
& -\left( \frac{1}{-\alpha +1} -\frac{1}{p_s^*}\right)\Vert c \Vert_{L^{\infty}}\vert \Omega\vert^{\frac{p_s^* +\alpha -1}{p_s^*}}S_{p}^{\frac{\alpha-1}{p}} \Vert u \Vert^{-\alpha+1}\\
\end{aligned}
\end{equation*}
Put 
\begin{equation*}
\xi(t) =b\left( \frac{1}{p\theta} -\frac{1}{p_s^*}\right)t^{p \theta} 
 -\left( \frac{1}{-\alpha +1} -\frac{1}{p_s^*}\right)\Vert c \Vert_{L^{\infty}}\vert \Omega\vert^{\frac{p_s^* +\alpha -1}{p_s^*}}S_{p}^{\frac{\alpha-1}{p}} t^{-\alpha+1}
\end{equation*}
$ \xi$ attains its minimum at $ t_{min}$ where
\begin{equation*}
t_{min} =\left( \frac{\frac{p_s^* +\alpha -1}{p_s^*}\Vert c \Vert_{L^{\infty}}\vert \Omega\vert^{\frac{p_s^* +\alpha -1}{p_s^*}}S_{p}^{\frac{\alpha-1}{p}}}{b\frac{p_s^* -p\theta}{p_s^*}}\right)^{\frac{1}{p\theta +\alpha -1}}
\end{equation*}
namely
{\footnotesize
\begin{equation*}
\xi(t_{min}) =-\left(\frac{p_s^* +\alpha -1}{p_s^*}\Vert u\Vert_{\infty}\vert \Omega\vert^{\frac{p_s^* +\alpha -1}{p_s^*}}S_{p}^{\frac{\alpha -1}{p}}\right)^{\frac{p\theta}{p\theta +\alpha -1}}\left( b\frac{p_s^* -p\theta}{p_s^*p\theta}\right)^{\frac{\alpha -1}{p\theta +\alpha -1}}\left( \frac{\alpha}{1 -\alpha}\right)
\end{equation*}
}
therefore when $ n \to \infty$ we find
{\footnotesize
\begin{equation*}
c \geq a\left(\frac{p_s^* -p}{pp_s^*}\right)\mu^{p} -\left(\frac{p_s^* +\alpha -1}{p_s^*}\Vert u\Vert_{\infty}\vert \Omega\vert^{\frac{p_s^* +\alpha -1}{p_s^*}}S_{p}^{\frac{\alpha -1}{p}}\right)^{\frac{p\theta}{p\theta +\alpha -1}}\left( b\frac{p_s^* -p\theta}{p_s^*p\theta}\right)^{\frac{\alpha -1}{p\theta +\alpha -1}}\left( \frac{\alpha}{1 -\alpha}\right)
\end{equation*}
}
then by \ref{e031} we find
{\tiny
\begin{equation*}
c \geq a\frac{s}{n}\left(\frac{a}{\gamma \lambda p_s^*}\right)^{\frac{n}{s p_s^*}}S_{p}^{\frac{n}{s p}} -\left(\frac{p_s^* +\alpha -1}{p_s^*}\Vert u\Vert_{\infty}\vert \Omega\vert^{\frac{p_s^* +\alpha -1}{p_s^*}}S_{p}^{\frac{\alpha -1}{p}}\right)^{\frac{p\theta}{p\theta +\alpha -1}}\left( b\frac{p_s^* -p\theta}{p_s^*p\theta}\right)^{\frac{\alpha -1}{p\theta +\alpha -1}}\left( \frac{\alpha}{1 -\alpha}\right)
\end{equation*}
}
And thid contraducts \ref{lev}.
\end{proof}
\section{Proof of theorem \ref{th1}}
Put $ \lambda^{*} =\min\{ \lambda_{*}, \lambda_{**}\}$\\
first we prove solutions in $ \mathbb{N}_{\lambda}^{+},$ by Vetalis convergence theorem we get 
\begin{equation}\label{v01}
\lim_{n \to \infty}\int_{\Omega}F(x,u_n^+)dx =\int_{\Omega}F(x,u_*)dx.
\end{equation}
on the other hand by using Minkowski and the Holder inequalities as $ n \to \infty,$ we find
\begin{equation*}
\int c(x)\vert u_n\vert^{-\alpha+1} =\int c(x)\vert u_n -u_* +u_*\vert^{-\alpha+1} \leq \Vert c \Vert_{q^{\prime}}\int \vert u_*\vert^{-\alpha +1} +o(1).
\end{equation*}
and 
\begin{equation*}
\int c(x)\vert u_*\vert^{-\alpha+1} =\int c(x)\vert u_* -u_n +u_n\vert^{-\alpha+1} \leq \Vert c \Vert_{q^{\prime}}\int \vert u_n\vert^{-\alpha +1} +o(1).
\end{equation*}
That is 
\begin{equation}\label{e02}
\int c(x)(u_n^+)^{-\alpha+1} =\int c(x)(u_*^+)^{-\alpha+1} +o(1)
\end{equation}
therefore we get
\begin{equation}\label{v03}
\lim_{n \to \infty}\int_{\Omega}c(x)(u_n^+)^{1 -\alpha} =\int_{\Omega}c(x)u_*dx
\end{equation}
Also by the weakly lower semi-continuity of the norm and \ref{v02} one has
\begin{equation}\label{v04}
\lim_{n \to \infty}\Vert u_n\Vert =\Vert u_*\Vert, \quad \textit{and} \quad \lim_{n \to \infty}\Vert u_n\Vert_{p} =\Vert u_*\Vert_{p}. 
\end{equation}
combining \ref{v01}, \ref{v03}, \ref{v04} and by lemma \ref{l6} we find
\begin{equation*}
(p +\alpha -1) \left(a\Vert u_* \Vert^p + \Vert u_* \Vert_{p}^{p}\right) + b(p \theta +\alpha -1)\Vert u_* \Vert^{p \theta} - \lambda q (q +\alpha -1) \int F(x,u_*) >0.
\end{equation*}
which means that $ u_* \in \mathcal{N}_{\lambda}^{+},$ also by combining \ref{v01}, \ref{v03}, \ref{v04} and by lemma \ref{l8} we get 
\begin{equation*}
\begin{aligned}
& \left(a +b\Vert u_* \Vert^{p(\theta -1)}\right) \int_{\mathbb{R}^{2n}} \frac{\vert u_*(x) - u_*(y) \vert^{p-2}(u_*(x) - u_*(y))(\phi(x) - \phi(y))}{\vert x-y \vert^{n + ps}}\\
& + \int_{\Omega}(u_*)^{p -1} \phi(x) dx  - \int_{\Omega} c(x) (u_*)^{-\alpha} \phi(x) - 
\lambda \int_{\Omega} f(x,u_*) \phi(x)dx \geq 0.
\end{aligned}
\end{equation*}
For any $ \phi \in X_0,$ then by replacing $ \phi$ by $ -\phi$ we get that $ u_*$ is a solution of problem \ref{p}, furthermore by \ref{l7} and Fatous lemma
we get that $ u_*$ is a positive solution to \ref{p0}??. Now we prove the existence of solutions in $ \mathcal{N}_{\lambda}^{-},$
Similarly to case in $ \mathcal{N}_{\lambda}^{+},$ by \ref{ek2} there exist a positive seqence $ (v_n)_n \subset \mathcal{N}_{\lambda}^{-}$ such that 
\begin{itemize}
\item $ J_{\lambda}(v_n) <m^- +\frac{1}{n},$
\item $ J_{\lambda}(v_n) \geq J_{\lambda}(v_n) -\frac{1}{n}\Vert v -v_n\Vert \quad \textit{for any}\quad v \in \mathbb{N}_{\lambda}^{-}.$
\end{itemize}
Due to the coercivity of $ J_{\lambda}$ on $  \mathbb{N}_{\lambda},$ $ v_n$ is bounded in $ X_0,$ so upto subsequence still denoted $ v_n$ there exist $ v_* >0$
such that 
\begin{itemize}
\item $ v_n \rightharpoonup v_*,$ converges weakly in $ X_0,$
\item $ v_n \longrightarrow v_*,$ strongly in $ L^{\eta}(\Omega)$ for $ \eta \in [1, p_s^*),$ therefore in $ L^{-\alpha +1}(\Omega),$
\item $ v_n(x) \longrightarrow v_*(x),$ a.e. in $ \Omega,$
\end{itemize}
as $ n \to \infty.$ By the weak lower semi-continuity and \ref{ek1} we obtain 
\begin{equation*}
J_{\lambda}(v_*) \leq \lim_{n \to \infty}\inf J_{\lambda}(v_n) =\inf_{u \in \mathcal{N}_{\lambda}^{-}}(u)
\end{equation*} 
we can see that $ v_* \neq 0$ in $ \Omega,$ we know that $ v_n \in \mathcal{N}_{\lambda}^{-},$ then by combining \ref{v01}, \ref{v03}, \ref{v04} and by 
lemma \ref{l6} we obtain
\begin{equation}
(p +\alpha -1) \left(a\Vert v_* \Vert^p + \Vert v_* \Vert_{p}^{p}\right) + b(p \theta +\alpha -1)\Vert v_* \Vert^{p \theta} - \lambda q (q +\alpha -1) \int F(x,v_*) <0.
\end{equation}
which means that $ v_* \in \mathcal{N}_{\lambda}^{-}.$ Now similarly to the first claim by applying \ref{l8}, and combining \ref{v01}, \ref{v03}, \ref{v04}
replacing $ u_n$ by $ v_n$ we find 
\begin{equation*}
\begin{aligned}
& \left(a +b\Vert v_* \Vert^{p(\theta -1)}\right) \int_{\mathbb{R}^{2n}} \frac{\vert v_*(x) - v_*(y) \vert^{p-2}(v_*(x) - v_*(y))(\phi(x) - \phi(y))}{\vert x-y \vert^{n + ps}}\\
& + \int_{\Omega}(v_*)^{p -1} \phi(x) dx  - \int_{\Omega} c(x) (v_*)^{-\alpha} \phi(x) - 
\lambda \int_{\Omega} f(x,v_*) \phi(x)dx \geq 0.
\end{aligned}
\end{equation*}
and since this last inequality holds for any $ \phi \in X_0$ then it hold for $ -\phi,$ which give us that $ v_*$ is a solutions of the problem \ref{p},
moreover using \ref{l7} and Fatous lemma we get that $ v_*$ is a solution to problem \ref{p0}, this completes the proof of theorem \ref{th1}. 
\section{Proof Of Theorem \ref{th2}}
in this section we prove the existence of solution of \ref{p0} regarding the critical case where $ q =p_s^*$\newline
Let
{\tiny
\begin{equation}
\lambda_{***} =\frac{1}{\gamma}\left( \frac{s}{n}\right)^{\frac{s p_s^*}{n}} a^{\frac{pp_s^*}{p}}\left\lbrace \left(\frac{p_s^* S_{p}^{\frac{1 -\alpha}{p}}}{(p_s^* +\alpha -1)\Vert u\Vert_{\infty}\vert \Omega\vert^{\frac{p_s^* +\alpha -1}{p_s^*}}}\right)^{\frac{p\theta}{p\theta +\alpha -1}}\left( b\frac{p_s^* -p\theta}{p_s^* p\theta}\right)^{\frac{1 -\alpha}{p\theta +\alpha -1}}\left( \frac{1 -\alpha}{\alpha}\right)\right\rbrace^{\frac{s p_s^*}{n}}
\end{equation}
}
starting with the case in $ \mathcal{N}_{\lambda}^{+}$
\begin{proof}
let $ \lambda^{**} =\min\{ \lambda_{*}, \lambda_{**}, \lambda_{***}\}$ where $ \lambda_{*}, \lambda_{**}$ given in lemma \ref{l1}, and \ref{l2},
by lemma \ref{l2} there exist sequence $ (u_n)_n \in \mathcal{N}_{\lambda}^{+}\cup\{ 0\},$ such that $ u_n \rightharpoonup u^*$ in $ X_0$ and satisfies conditions \ref{ek1}, \ref{ek2} note that $ c_{\lambda} >0,$ since $ \lambda <\lambda^{***},$ moreover if  $ u \in  \mathcal{N}_{\lambda}^{+},$ then for $ \lambda >0$ by corollary \ref{c1} and  $ J_{\lambda}(0) =0$ we get that
\begin{equation*}
m^{+} =\inf_{u \in \mathcal{N}_{\lambda}^{+}}J_{\lambda}(u) <0.
\end{equation*}
Hence by \ref{ek1} we obtain
\begin{equation*}
\lim_{n \to \infty}J_{\lambda}(u_n) =m^+  =\inf_{u \in \in \mathcal{N}_{\lambda}^{+}\cup\{ 0\}} J_{\lambda}(u) <0 \quad \textit{as} \quad n \to \infty.
\end{equation*}
which means that $ (u_n)_n \in \mathcal{N}_{\lambda}^{+}.$ Therefore replacing $ c_{\lambda}$ by $ m^+$ in proposition \ref{pr1} we get that there exist $ u^* \in X_0$ such that up to subseqence $ u_n \longrightarrow u^*$ strongly in $ X_0,$ so we can conclude that $ u^* \in \mathcal{N}_{\lambda},$
Also since $ u_n \rightharpoonup u^*$ in $ X_0$ then by lemma \ref{l6} and the weakly lower semi-continuity of the norm we get 
\begin{equation*}
(p +\alpha -1) \left(a\Vert u^* \Vert^p + \Vert u^* \Vert_{p}^{p}\right) + b(p \theta +\alpha -1)\Vert u^* \Vert^{p \theta} 
- \lambda q (q +\alpha -1) \int F(x,u^*) >0.
\end{equation*}
which means that $ u^* \in \mathcal{N}_{\lambda}^{+},$ moreover by the continuity of $ J_{\lambda}$ on $ \mathcal{N}_{\lambda}^{+}$ then $ J_{\lambda}$ attains its minimum $ m^+$ at $ u^*,$ on the other hand by lemma \ref{l8} and Fatous lemma we have 
\begin{equation}\label{eee60}
\begin{aligned}
& \left(a +b\Vert u^* \Vert^{p(\theta -1)}\right) \int_{\mathbb{R}^{2n}} \frac{\vert u^*(x) - u^*(y) \vert^{p-2}(u^*(x) - u^*(y))(\phi(x) - \phi(y))}{\vert x-y \vert^{n + ps}}\\
& + \int_{\Omega}((u^*)^+)^{p -1} \phi(x) dx  - \int_{\Omega} c(x) ((u^*)^+)^{-\alpha} \phi(x) - 
\lambda \int_{\Omega} f(x,u^*) \phi(x)dx \geq 0.
\end{aligned}
\end{equation}
we can easly see that this last inequaliy holds for any $ \phi \in X_0,$ moreover \newline
$ c(x) ((u^*)^+)^{-\alpha} \phi(x) \in L^{1}(\Omega)$ yields that $ (u^*)^+$ is a weak solution to \ref{p}, also by the weak lower semi continiuity, lemma \ref{l2} and \ref{l5} we get that $ u^* \neq 0$ then $ u^*$ is nontrivial solution to problem \ref{p}. on the other hand replacing $ \phi $ by $ (u^*)^-$ in \ref{weak} we get that $ \Vert (u^*)^-\Vert =0,$ that is  $ u^* \geq 0.$ Applying the maximum priciple we get that $ u^*$ is a positive solution to problem \ref{p0}.
\begin{corollary}
in this corollaary we will demonstrate that $ u^*$ is a local minimizer for $ J_{\lambda},$ by lemma \ref{l1} there exist $ t_1(u) <t_{max}(u)$ such that 
$ t_1(u^*)u^* \in \mathcal{N}_{\lambda}^{+},$ and since $ u^* \in \mathcal{N}_{\lambda}^{+}$ then $ t_1(u^*) =1 <t_{max}(u),$ hence there exist $ \epsilon >0$ and an open ball $ B(\epsilon)\subset ]1,\infty)$ such that $ t_{max}(u^*) \in B$ that is $ \Vert t_{max}(u^*) -t\Vert < \epsilon$ and by continuity of $ t_{max}$ for some $ \delta >0,$ we can take $ t =t(u)$ where $ \Vert u\Vert <\delta $ yields that $ t_{max}(u^* -u) \in B(\epsilon) \subset ]1,\infty)$
hence $  t_{max}(u_0 -u) >\epsilon +1.$ Now from lemma \ref{l4} for some $ \delta^{\prime} >0$ there exist a continuous function $ \xi:B(0,\delta^{\prime} \to \mathbb{R}^{+})$ such that $ \xi(u)(u^* -u) \in \mathcal{N}_{\lambda}^{+}$ and $ \xi(0) =0,$ hence for $ \delta^{\prime\prime} =\min\{\delta,\delta^{\prime}\}$ 
we have that $ t_1(u^* -u) = \xi(u) < 1 +\epsilon <t_{max}(u^* -u)$ for all $\Vert u\Vert <\delta^{\prime\prime}.$ further we have that 
$ t_{max}(u^* -u) >1,$ yields $ J_{\lambda}(u^*) \leq J_{\lambda}(t_1((u^* -u))(u^* -u)) \leq J_{\lambda}(u^* -u)$ for $ \Vert u\Vert <\delta^{\prime\prime}.$  
this finshes the proof
\end{corollary}
\end{proof}
Proving existence in  $ \mathcal{N}_{\lambda}^{-},$ assuming that $ 0 \in \Omega^{+},$ and let $ \sigma$ be the test function defined for sufficently 
small $ \delta$ as
\begin{equation}\label{eee10}
0 \leq \sigma(x) \leq 1 \,\,\, \textit{in} \,\,\, \Omega^{+}, \,\,\, \sigma(x) =1 \,\,\, \textit{on} \,\,\, B_{\frac{\delta}{2}}(0) \,\,\, \textit{and}\,\,\,
\sigma(x) =0 \,\,\, \textit{on} \,\,\, \mathbb{R}^{n}\setminus B_{\sigma}(0),
\end{equation}
Also we put $ u_{\epsilon, \sigma} =\sigma u_{\epsilon} \in X_0,$ where $ u_{\epsilon}$ given in \ref{eee1} and as in \cite{tarantello1992nonhomogeneous} we define
\begin{equation*}
U_1 =\left\lbrace u \in X_0\setminus\{ 0\}: \frac{1}{\Vert u\Vert} t_{2}\left(\frac{u}{\Vert u\Vert}\right) > 1\right\rbrace \cup \{ 0\},
\end{equation*}
and
\begin{equation*}
U_2 =\left\lbrace u \in X_0\setminus\{ 0\}: \frac{1}{\Vert u\Vert} t_{2}\left(\frac{u}{\Vert u\Vert}\right) < 1\right\rbrace,
\end{equation*}
by lemma \ref{l1} there exist $ t_2$ such that if $ u \in \mathcal{N}_{\lambda}^{-}$ we get that $ t_2(u) =1$ that is 
$ \Vert u\Vert \frac{1}{\Vert u\Vert} u =u \in \mathcal{N}_{\lambda}^{-}$ yields $ t_2(\frac{u}{\Vert u\Vert}) =\Vert u\Vert$ means $ \frac{1}{\Vert u\Vert} t_2(\frac{u}{\Vert u\Vert}) =1$ that is to say
\begin{equation*}
\mathcal{N}_{\lambda}^{-} =\left\lbrace u \in X_0\setminus\{ 0\}: \frac{1}{\Vert u\Vert} t_{2}\left(\frac{u}{\Vert u\Vert}\right) = 1\right\rbrace
\end{equation*}
and since $ 1 < t_{max}(u) <t_{2}(u)$ when $ u \in \mathcal{N}_{\lambda}^{+}$ means that $ \mathcal{N}_{\lambda}^{+} \subset U_1$ hence by lemma \ref{l6} 
we get that $ u^* \in \mathcal{N}_{\lambda}^{+}.$
\begin{lemma}\label{l9}
let $ \lambda \in (0 ,\lambda^{**})$ and $ v_*$ be a local minimizer in $ X_0$ for $ J_{\lambda}$ then for any $\epsilon >0 $ and any test function $ \sigma$ defined in \ref{eee10} there exist $ l_0 >0$ such that $ u_0 +l_0u_{\epsilon,\sigma} \in U_2$  
\end{lemma}
\begin{proof}
there exist $ C_6 >0$ such that $ 0 <t_2(\frac{u_0 +lu_{\epsilon,\sigma}}{\Vert u_0 +lu_{\epsilon,\sigma}\Vert}) < C_6$ suppose the opposite let $ (l_n)_n \in \mathbb{R}^+$ such that $ l_n \to \infty$ with $ t_2(v_n) \longrightarrow \infty$ where $ v_n =\frac{u_0 +lu_{\epsilon,\sigma}}{\Vert u_0 +lu_{\epsilon,\sigma}\Vert}$ as $ n \to \infty.$ By the Lebesgue dominated convergence theorem, we get 
\begin{equation*}
\begin{aligned}
\lim_{n \to \infty}\int_{\Omega}F(x,v_n)dx &=\lim_{n \to \infty}\frac{1}{\Vert u_0 +l_nu_{\epsilon,\sigma}\Vert^{p_s^*}}\int_{\Omega}F(x,u_0 +l_nu_{\epsilon,\sigma})dx\\
& =\lim_{n \to \infty}\frac{1}{\Vert \frac{u_0}{l_n}  +u_{\epsilon,\sigma}\Vert^{p_s^*}}\int_{\Omega}F(x,\frac{u_0}{l_n} +u_{\epsilon,\sigma})dx\\
& =\frac{1}{\Vert u_{\epsilon, \sigma\Vert}}\int_{\Omega}F(x, u_{\epsilon,\sigma})dx\\
\end{aligned}
\end{equation*}
and by \ref{e10} if $ u \in \mathcal{N}_{\lambda}^{-}$ we obtain $ \int_{\Omega} F(x,u)dx >0$ from this follows that funcational $ J_{\lambda}$ is unbounded below at $ t_2(v_n)v_n$ which contrducts the coercivity of it that is to say that as  $ n \to \infty$
\begin{equation*}
\begin{aligned}
J_{\lambda}(t_2(v_n)v_n) &=  \frac{t_2(v_n)^P}{p} \left(a \Vert v_n \Vert^p + \Vert v_n \Vert^{p}_{p}\right) + \frac{b t_2(v_n)^{p \theta}}{p \theta}\Vert v_n \Vert^{p \theta}\\ 
&-\frac{t_2(v_n)^{1-\alpha}}{-\alpha+1} \int c(x)v_n^{-\alpha+1} - \lambda t_2(v_n)^q \int F(x,v_n) \longrightarrow \infty.
\end{aligned}
\end{equation*}
Put $ l_0^{p} =\frac{\vert c^p -\Vert u\Vert^{p} \vert}{\Vert u_{\epsilon,\sigma}\Vert^{p}} +c^{\prime} $ where $ c^{\prime} >0$
{\tiny
\begin{equation*}
\begin{aligned}
&\Vert u_0 +l_0u_{\epsilon,\sigma}\Vert^{p}\\
&=\Vert u_0\Vert^{p} +l_0^{p}\Vert u_{\epsilon,\sigma}\Vert^{p} +2l_0\int_{\mathbb{R}^{2n}}\frac{\vert u_0(x) -u_0(y)\vert^{p -2}(u_0(x) -u_0(y))(u_{\epsilon,\sigma}(x) -u_{\epsilon,\sigma}(y))}{\vert x -y\vert^{n +ps}}dxdy\\
&>\Vert u_0\Vert^{p} +l_0^{p}\Vert u_{\epsilon,\sigma}\Vert^{p} > \Vert u_0\Vert^{p} +\vert c^p -\Vert u\Vert^{p}\vert \geq c^{p} >\left( t_2(\frac{u_0 +lu_{\epsilon,\sigma}}{\Vert u_0 +l_0u_{\epsilon,\sigma}\Vert})\right)^{p}
\end{aligned}
\end{equation*}
}
this means that $  u_0 +l_0u_{\epsilon,\sigma} \in U_2.$
\end{proof}
\begin{lemma}\label{l10}
let $ \lambda <\lambda^{**}$ and $ u^*$ local minimum of $ J_{\lambda}$ in $ X_0$ then for all $ r >0$ and $ \sigma$ as in \ref{eee10} there exist $\epsilon^{\prime} >0$ depending on $ r, \sigma$ and there exist $ \lambda^{***} >0$ such that
\begin{equation*}
J_{\lambda}(u_0 +ru_{\epsilon,\sigma}) <c_{\lambda},\,\,\, \textit{for any} \,\,\, \epsilon \in (0, \epsilon_{0})\,\,\, \textit{and}\,\,\, \lambda < \lambda^{***}
\end{equation*} 
for $ b \in (0 ,\epsilon_{0}^{m})$ with $ m > \frac{n-ps}{p -1},$
\end{lemma}
\begin{proof}
we have
{\footnotesize
\begin{equation*}
\begin{aligned}
&J_{\lambda}(u_0 +ru_{\epsilon,\sigma})=\frac{1}{p} \left(a \Vert u_0 +ru_{\epsilon,\sigma}\Vert^p + \Vert u_0 +ru_{\epsilon,\sigma} \Vert^{p}_{p}\right) + \frac{b}{p \theta}\Vert u_0 +ru_{\epsilon,\sigma} \Vert^{p \theta} \\
& -\frac{1}{-\alpha+1} \int c(x)(u_0 +ru_{\epsilon,\sigma})^{-\alpha+1} - \lambda\int F(x,u_0 +ru_{\epsilon,\sigma})dx\\
& =\frac{a}{p} \left(\Vert u_0\Vert^{p} +r^{p}\Vert u_{\epsilon,\sigma}\Vert^{p} +2r\int_{\mathbb{R}^{2n}}\frac{\vert u_0(x) -u_0(y)\vert^{p -2}(u_0(x) -u_0(y))(u_{\epsilon,\sigma}(x) -u_{\epsilon,\sigma}(y))}{\vert x -y\vert^{n +ps}}dxdy \right)\\
& +\frac{1}{p}\left( \Vert u_0\Vert^{p}_{p} +r^{p}\Vert u_{\epsilon,\sigma}\Vert^{p}_{p} +2r\int_{\Omega}u_0^{p-1}u_{\epsilon,\sigma}dx\right)\\
& +\frac{b}{p\theta}\left( \Vert u_0\Vert^{p} +r^{p}\Vert u_{\epsilon,\sigma}\Vert^{p} +2r\int_{\mathbb{R}^{2n}}\frac{\vert u_0(x) -u_0(y)\vert^{p -2}(u_0(x) -u_0(y))(u_{\epsilon,\sigma}(x) -u_{\epsilon,\sigma}(y))}{\vert x -y\vert^{n +ps}}dxdy \right)^{\theta}\\
& -\frac{1}{-\alpha+1} \int c(x)(u_0 +ru_{\epsilon,\sigma})^{-\alpha+1} - \lambda\int F(x,u_0 +ru_{\epsilon,\sigma})dx\\
\end{aligned}
\end{equation*}
}
Now using the  following ineauality as in \cite{fiscella2019nehari} which holds for any $ c, d\geq 0$ and $ \theta >1,$ with $ C(\theta) >0$ 
\begin{equation*}
\left( c +d\right)^{\theta} \leq c^{\theta} +C(\theta)\left( c^{\theta} +d^{\theta}\right) +\theta c^{\theta -1}d.
\end{equation*}
we get
\footnotesize
\begin{equation*}
\begin{aligned}
&+\frac{b}{p\theta}\left( \Vert u_0\Vert^{p} +r^{p}\Vert u_{\epsilon,\sigma}\Vert^{p} +2r\int_{\mathbb{R}^{2n}}\frac{\vert u_0(x) -u_0(y)\vert^{p -2}(u_0(x) -u_0(y))(u_{\epsilon,\sigma}(x) -u_{\epsilon,\sigma}(y))}{\vert x -y\vert^{n +ps}}dxdy \right)^{\theta}\\
&\leq \frac{b}{p\theta}\Vert u_0\Vert^{p} +\frac{2rb}{p} \Vert u_0\Vert^{p(\theta -1)}\int_{\mathbb{R}^{2n}}\frac{\vert u_0(x) -u_0(y)\vert^{p -2}(u_0(x) -u_0(y))(u_{\epsilon,\sigma}(x) -u_{\epsilon,\sigma}(y))}{\vert x -y\vert^{n +ps}}dxdy\\
& +bC_{\theta} r^{p\theta}\Vert u_{0}\Vert^{2\theta} +bD_{\theta} \Vert u_{\epsilon,\sigma}\Vert^{p\theta}\\
\end{aligned}
\end{equation*}
then 
{\tiny
\begin{equation*}
\begin{aligned}
&J_{\lambda}(u_0 +ru_{\epsilon,\sigma}) \leq  J_{\lambda}(u_0)+\frac{r^{p}}{p}\left( a\Vert u_{\epsilon,\sigma}\Vert +\Vert u_{\epsilon,\sigma}\Vert^{p}_{p}\right) +bC_{\theta}\Vert u_{\epsilon,\sigma}\Vert^{p\theta} +b D_{\theta} r^{p\theta}\\
& \frac{2r}{p}\left\{\left(a +b\Vert u_0\Vert^{p(\theta -1)}\right)\int_{\mathbb{R}^{2n}}\frac{\vert u_0(x) -u_0(y)\vert^{p -2}(u_0(x) -u_0(y))(u_{\epsilon,\sigma}(x) -u_{\epsilon,\sigma}(y))}{\vert x -y\vert^{n +ps}}dxdy +\int_{\Omega}u_0^{p-1} u_{\epsilon,\sigma}dx\right\}\\
& -\frac{1}{1 -\alpha}\int c(x)(u_0 +ru_{\epsilon,\sigma})^{-\alpha+1} -u_0^{-\alpha +1}dx -\lambda\int F(x,u_0 +ru_{\epsilon,\sigma}) -F(x,u_0)dx\\
\end{aligned}
\end{equation*}
}
Since $ u_0$ is a weak solution to the problem we obtain
\begin{equation}\label{eee20}
\begin{aligned}
& J_{\lambda}(u_0 +ru_{\epsilon,\sigma})\leq  J_{\lambda}(u_0)+\frac{r^{p}}{p}\left( a\Vert u_{\epsilon,\sigma}\Vert +\Vert u_{\epsilon,\sigma}\Vert^{p}_{p}\right) +bC_{\theta}\Vert u_{0}\Vert^{p\theta} +b D_{\theta} r^{p\theta}\Vert u_{\epsilon,\sigma}\Vert^{p\theta}\\
& -\frac{1}{1 -\alpha}\int c(x)(u_0 +ru_{\epsilon,\sigma})^{-\alpha+1} -u_0^{-\alpha +1} -\frac{2r(1 -\alpha)}{p}u_0^{-\alpha}u_{\epsilon,\sigma}dx\\
& -\lambda\int F(x,u_0 +ru_{\epsilon,\sigma}) -F(x,u_0) -\frac{2r}{p}f(x, u_0)u_{\epsilon,\sigma} dx\\
\end{aligned}
\end{equation}
by \cite{sun2011exact}, \cite{fiscella2019nehari} and \cite{mosconi2016brezis} we have
\begin{equation*}
\begin{aligned}
&\int c(x)\{(u_0 +rU_{\epsilon,\sigma})^{-\alpha+1} -u_0^{-\alpha +1}\} \leq \Vert c\Vert_{\infty}r^{-\alpha +1}  \int U_{\epsilon,\sigma}^{1 -\alpha}dx\\ 
&\leq  \Vert c\Vert_{\infty}r^{-\alpha +1} \epsilon^{\frac{n -ps}{p -1}(1 -\alpha)} \int 
\left( \epsilon^{p^{\prime}} +\vert x\vert^{p^{\prime}}\right)^{-\frac{n -sp}{p}(1 -\alpha)}dx \\
\end{aligned}
\end{equation*}
\begin{equation*}
\int c(x)\{(u_0 +rU_{\epsilon,\sigma})^{-\alpha+1} -u_0^{-\alpha +1}\} =O(\epsilon^{\frac{n -ps}{p -1}(1 -\alpha)})
\end{equation*}
and by Cauchy Schwartz inequality we get 
{\scriptsize
\begin{equation*}
\int u_0^{-\alpha}u_{\epsilon,\sigma}dx =\int u_0^{-\alpha}u_{\epsilon}dx \leq c\int \left( u_{\epsilon}^{2}dx\right)^{\frac{1}{2}} = C \epsilon^{\frac{n -ps}{p -1}}\left( \int \left( \epsilon^{p^{\prime} +\vert x\vert^{p^{\prime}}}\right)^{-2\frac{n -ps}{p}}\right)^{\frac{1}{2}} =O(\epsilon^{\frac{n -ps}{p -1}}).
\end{equation*}
}
Hence we get
\begin{equation}\label{eee21}
-\frac{1}{1 -\alpha}\int c(x)(u_0 +ru_{\epsilon,\sigma})^{-\alpha+1} -u_0^{-\alpha +1} -\frac{2r(1 -\alpha)}{p}u_0^{-\alpha}u_{\epsilon,\sigma}dx =O(\epsilon^{\frac{n -ps}{p -1}})
\end{equation}
and by \ref{fd2}, \ref{fd} there exist $ C_{7} >0$ such that
\begin{equation*}
\int F(x,u_0 +ru_{\epsilon,\sigma}) -F(x,u_0) -\frac{2r}{p}f(x, u_0)u_{\epsilon,\sigma} dx \geq r^{p_s^*}\int F(x,u_{\epsilon,\sigma})dx +C_{7} r^{p_s^* -1}\int f(x, u_{\epsilon,\sigma})u_0 dx
\end{equation*}
then combining \ref{eee20}, \ref{eee21} and denoting $ \Vert u_0\Vert =R$ we get
\begin{equation}\label{eee50}
\begin{aligned}
& J_{\lambda}(u_0 +ru_{\epsilon,\sigma})\leq J_{\lambda}(u_0)+\frac{r^{p}}{p}\left( a\Vert u_{\epsilon,\sigma}\Vert +\Vert u_{\epsilon,\sigma}\Vert^{p}_{p}\right) +bC_{\theta}R^{p\theta} +b D_{\theta} r^{p\theta}\Vert u_{\epsilon,\sigma}\Vert^{p\theta}\\
& -\lambda\left( r^{p_s^*}\int F(x,u_{\epsilon,\sigma})dx +C_{7} r^{p_s^* -1}\int f(x, u_{\epsilon,\sigma})u_0 dx\right) +O(\epsilon^{\frac{n -ps}{p -1}})
\end{aligned}
\end{equation}
Let $ T$ the function defined as 
{\scriptsize
\begin{equation}\label{eee30}
T(t) =\frac{t^{p}}{p}\left( a\Vert u_{\epsilon,\sigma}\Vert +\Vert u_{\epsilon,\sigma}\Vert^{p}_{p}\right) +b D_{\theta} t^{p\theta}\Vert u_{\epsilon,\sigma}\Vert^{p\theta} -\lambda t^{p_s^*}\int F(x,u_{\epsilon,\sigma})dx -\lambda C_{7}^{\prime} t^{p_s^* -1}\int f(x, u_{\epsilon,\sigma}) dx
\end{equation}
}
we can see that if $ x \in \Omega^{+},$ $ T(t) \longrightarrow -\infty$ as $ t \to \infty.$ and $ T(t) \xrightarrow{>} 0$ as $ \xrightarrow{>}0$ then by Roll theoem there exist $ t_{\epsilon} >0$
such that 
\begin{equation*}
T(t_{\epsilon}u_{\epsilon,\sigma}) =\sup_{t \geq 0}T(t), \quad \textit{and} \quad \frac{d}{dt}T(t_{\epsilon}u_{\epsilon,\sigma}) =0,
\end{equation*}
and
\begin{equation*}
\begin{aligned}
& T^{\prime}(t_{\epsilon}) =t_{\epsilon}^{p -1}\left( a\Vert u_{\epsilon,\sigma}\Vert^{p} +\Vert u_{\epsilon,\sigma}\Vert^{p}_{p}\right) +b p\theta(p\theta -1)
 D_{\theta} t_{\epsilon}^{p\theta -1}\Vert u_{\epsilon,\sigma}\Vert^{p\theta} -\lambda p_s^* t_{\epsilon}^{p_s^* -1}\int F(x,u_{\epsilon,\sigma})dx\\
& -\lambda (p_s^* -1)C_{5}^{\prime} t_{\epsilon}^{p_s^* -2}\int f(x, u_{\epsilon,\sigma}) dx
\end{aligned}
\end{equation*}
by that fact that $ T^{\prime}(t_{\epsilon}) >0,$ for $ t \in (0, t_{\epsilon})$ we obtain
\begin{equation*}
\begin{aligned}
& (p -1)\left( a\Vert u_{\epsilon,\sigma}\Vert +\Vert u_{\epsilon,\sigma}\Vert^{p}_{p}\right) <\lambda p_s^* (p_s^* -1)t_{\epsilon}^{p_s^* -p}\int F(x,u_{\epsilon,\sigma})dx\\
& +\lambda (p_s^* -2)(p_s^* -1)C_{5}^{\prime} t_{\epsilon}^{p_s^* -p-1}\int f(x, u_{\epsilon,\sigma}) dx -b p\theta (p\theta -1)D_{\theta} t_{\epsilon}^{p\theta -p}\Vert u_{\epsilon,\sigma}\Vert^{p\theta}
\end{aligned}
\end{equation*}
it mean that $ t_{\epsilon}$ is uniformly bounded from bellow that is to say there exist $ t_3$ such that $ 0 <t_3 \leq t_{\epsilon}.$ 
on the other hand  for $ t >t_{\epsilon},$ since $ T^{\prime}(t_{\epsilon}) <0,$ we have that
\begin{equation*}
\begin{aligned}
& b p\theta (p\theta -1) D_{\theta} \Vert u_{\epsilon,\sigma}\Vert^{p\theta} >\lambda p_s^* (p_s^* -1)t_{\epsilon}^{p_s^* -p\theta}\int F(x,u_{\epsilon,\sigma})dx\\
& +\lambda (p_s^* -2)(p_s^* -1)C_{5}^{\prime} t_{\epsilon}^{p_s^* -p\theta -1}\int f(x, u_{\epsilon,\sigma}) dx -(p -1)t_{\epsilon}^{p -p\theta}\left( a\Vert u_{\epsilon,\sigma}\Vert +\Vert u_{\epsilon,\sigma}\Vert^{p}_{p}\right)
\end{aligned}
\end{equation*}
this means that there exist $ t_4 > 0$ such that $ t_{\epsilon} \leq t_4$\newline
Also following the steps in \cite{servadei2015brezis} and \cite{mosconi2016brezis} we obtain the folowing estimates
\begin{equation}\label{es01}
\begin{aligned}
&\int_{\mathbb{R}^n} \vert u_{\epsilon,\sigma}\vert^{p_s^*}dx =\int_{\mathbb{R}^n} \vert u_{\epsilon}\vert^{p_s^*} +\int_{\mathbb{R}^n} \left( 
\vert \sigma(x)\vert^{p_s^*} -1 \right)\vert u_{\epsilon}\vert^{p_s^*}\\
& = S_{p}^{\frac{n}{ps}} +C_6\epsilon^{ \frac{n}{p -1}}\int_{B_{\delta}^{c}} \left( \epsilon^{p^{\prime}} +\vert x\vert^{p^{\prime}}\right)^{-n}\\
& =S_{p}^{\frac{n}{ps}} +O(\epsilon^{\frac{n}{p -1}}).
\end{aligned}
\end{equation}
and
\begin{equation}\label{es03}
\begin{aligned}
&\Vert u_{\epsilon,\sigma}\Vert^{p}=\Vert \sigma u_{\epsilon}\Vert^{p} \leq S_{p}^{\frac{n}{ps}} +O(\epsilon^{\frac{n -ps}{p -1}}),\\
& \Vert u_{\epsilon,\sigma}\Vert^{p\theta} \leq S_{p}^{\frac{n\theta}{ps}} +O(\epsilon^{\frac{n -ps}{p -1}})\\
&  \Vert u_{\epsilon,\sigma}\Vert^{p}_{p} = \int_{\Omega}\vert u_{\epsilon,\sigma}\vert^{p}dx = \int_{\Omega}\vert \sigma u_{\epsilon}\vert^{p}dx \leq \epsilon^{\frac{n -ps}{p(p -1)}}\int \left( \epsilon^{p^{\prime}} +\vert x\vert^{p^{\prime}}\right)^{-\frac{n -ps}{p}} =O(\epsilon^{\frac{n -ps}{p -1}})\\
\end{aligned}
\end{equation}
Also  by \cite{mosconi2016brezis} lemma 2.2 we have
\begin{equation}\label{es02}
\int_{\Omega} u_{\epsilon}^{p_s^* -1} = \epsilon^{\frac{n -sp}{p}}\int u(x)^{p_s^* -1} =O(\epsilon^{\frac{n -sp}{p}}).
\end{equation}
Combining \ref{eee20}, \ref{eee30} and \ref{slc} we get
\begin{equation*}
\begin{aligned}
& J_{\lambda}(u_0 +ru_{\epsilon,\sigma}) \leq \sup_{t \geq 0}\left\lbrace \frac{t^{p}}{p}\left( a\Vert u_{\epsilon,\sigma}\Vert^{p} +\Vert u_{\epsilon,\sigma}\Vert^{p}_{p}\right) -\lambda t^{p_s^*}\int F(x,u_{\epsilon,\sigma})dx\right\rbrace \\
&-\lambda C^{\prime}_{7} t^{p_s^* -1}\int f(x, u_{\epsilon,\sigma}) dx +bC_{\theta}R^{p\theta} +b D_{\theta} t^{p\theta}\Vert u_{\epsilon,\sigma}\Vert^{p\theta} +O(\epsilon^{\frac{n -ps}{p -1}})\\
\end{aligned}
\end{equation*} 
Also by \ref{fd} and \ref{es02} there exist $C_8 >0$ such that 
\begin{equation}\label{es04}
f(x,u_{\epsilon,\sigma}) \geq C_{8} \Vert u_{\epsilon,\sigma}\Vert^{p_s^* -1} =O(\epsilon^{\frac{n -sp}{p}})
\end{equation}
and by \ref{fd1} and \ref{es01} there exist $ C_9 >0$ such that 
\begin{equation}\label{es05}
F(x,u_{\epsilon,\sigma}) \geq C_{9}\Vert u_{\epsilon,\sigma}\Vert^{p_s^*} = C_{9}  S_{p}^{\frac{n}{ps}} +O(\epsilon^{\frac{n}{p -1}}).
\end{equation}
therefore by combining \ref{es04}, \ref{es05} and \ref{es03} we obtain
\begin{equation*}
\begin{aligned}
&\sup_{t \geq 0}\left\lbrace \frac{t^{p}}{p}\left( a\Vert u_{\epsilon,\sigma}\Vert^{p} +\Vert u_{\epsilon,\sigma}\Vert^{p}_{p}\right) -\lambda t^{p_s^*}\int F(x,u_{\epsilon,\sigma})dx\right\rbrace\\
&=\frac{s}{n}\left( \frac{1}{\lambda p_s^*}\right)^{\frac{p}{p_s^* -p}}\left(  a\Vert u_{\epsilon,\sigma}\Vert^{p} +\Vert u_{\epsilon,\sigma}\Vert^{p}_{p}\right)^{\frac{p_s^*}{p_s^* -p}}\left( \int F(x,u_{\epsilon,\sigma})dx\right)^{-\frac{p}{p_s^* -p}}\\
&\leq \frac{s}{n}\left( \frac{1}{\lambda p_s^*}\right)^{\frac{p}{p_s^* -p}}\left( a S_{p}^{\frac{n}{ps}} +O(\epsilon^{ \frac{n -ps}{p -1}}) +O(\epsilon^{\frac{n}{p -1}})\right)^{\frac{p_s^*}{p_s^* -p}}\left(C_{9} S_{p}^{\frac{n}{ps}} +O(\epsilon^{ \frac{n}{p -1}})\right)^{-\frac{p}{p_s^* -p}}\\
&= \frac{s}{n}\left( \frac{1}{\lambda p_s^*}\right)^{\frac{p}{p_s^* -p}}\left(  S_{p}^{\frac{n}{ps}} +O(\epsilon^{ \frac{n -ps}{p -1}}) +O(\epsilon^{\frac{n}{p -1}})\right)^{\frac{n}{ps}}\left(C_{9} S_{p}^{\frac{n}{ps}} +O(\epsilon^{ \frac{n}{p -1}})\right)^{-\frac{n}{sp_s^*}}\\
&= \frac{s}{n} a^{\frac{n}{ps}} \left( \frac{1}{C_{9}  \lambda p_s^*}\right)^{\frac{n}{s p_s^*}} \left( S_{p}^{\frac{n}{ps}} +O(\epsilon^{ \frac{n -ps}{p -1}})\right)^{\frac{n}{ps}}\left(S_{p}^{\frac{n}{ps}} +O(\epsilon^{ \frac{n}{p -1}})\right)^{-\frac{n}{sp_s^*}}\\
&=\frac{s}{n} a^{\frac{n}{ps}} (p_s^* \lambda)^{-\frac{n}{sp_s^*}} S_{p}^{\frac{n}{ps}}\left( \frac{1}{C_9}\right)^{\frac{n}{sp_s^*}} +O(\epsilon^{\frac{n -ps}{p -1}})
\end{aligned}
\end{equation*}
for this to be less then the first term in  $ c_{\lambda}$ it must be $  (C_9)^{-\frac{n}{sp_s^*}} < \gamma^{-\frac{n}{sp_s^*}},$ and by \ref{fd1} we can choose 
$ C_9$ sufficiently large. Now taking $ b =\epsilon^{m}, m > \frac{n -ps}{p -1} $ and using \ref{eee50} we find
\begin{equation*}
\begin{aligned}
& J_{\lambda}(u_0 +ru_{\epsilon,\sigma}) \leq \frac{s}{n} a^{\frac{n}{ps}} (p_s^* \lambda)^{-\frac{n}{sp_s^*}} S_{p}^{\frac{n}{ps}}C_9^{-\frac{n}{sp_s^*}} +O(\epsilon^{\frac{n -ps}{p -1}}) -\lambda C_5^{\prime} t_3^{p_s^* -1}\int f(x, u_{\epsilon,\sigma}) dx \\
&+bC_{\theta}R^{p\theta} +b D_{\theta} t_4^{p\theta}\Vert u_{\epsilon,\sigma}\Vert^{p\theta} +O(\epsilon^{\frac{n -ps}{p -1}})\\
& \leq \frac{s}{n} a^{\frac{n}{ps}} (p_s^* \lambda)^{-\frac{n}{sp_s^*}} S_{p}^{\frac{n}{ps}}C_9^{-\frac{n}{sp_s^*}} +O(\epsilon^{\frac{n -ps}{p -1}})\\
& -\lambda C^{\prime\prime}_{5} O(\epsilon^{\frac{n -sp}{p}}) +\epsilon^{m}C_{\theta}R^{p\theta} +\epsilon^{m} D_{\theta} t_4^{p\theta}(S_{p}^{\frac{n\theta}{ps}} +O(\epsilon^{\frac{n -ps}{p -1}})) +O(\epsilon^{\frac{n -ps}{p -1}})\\
& \leq \frac{s}{n} a^{\frac{n}{ps}} (p_s^* \lambda)^{-\frac{n}{sp_s^*}} S_{p}^{\frac{n}{ps}}C_9^{-\frac{n}{sp_s^*}} +C^{\prime\prime\prime}\epsilon^{\frac{n -ps}{p -1}} -\lambda C^{\prime\prime}_{5}\epsilon^{\frac{n -sp}{p}}\\
& \leq \frac{s}{n} a^{\frac{n}{ps}} (p_s^* \lambda)^{-\frac{n}{sp_s^*}} S_{p}^{\frac{n}{ps}}C_9^{-\frac{n}{sp_s^*}} -\left( \lambda C^{\prime\prime}_{5}\epsilon^{\frac{n -sp}{p}} -C^{\prime\prime\prime}\epsilon^{\frac{n -ps}{p -1}}\right).
\end{aligned}
\end{equation*}
Now we choose $ \epsilon^{\prime} >0$ small enogh such that for all $ \epsilon <\epsilon^{\prime},$ we get
\begin{equation*}
\lambda C^{\prime\prime}_{5}\epsilon^{\frac{n -sp}{p}} -C^{\prime\prime\prime}\epsilon^{\frac{n -ps}{p -1}} = \epsilon^{{\frac{n -ps}{p}}} 
\left( \lambda  C^{\prime\prime}_{5} -C^{\prime\prime\prime}\epsilon^{\frac{n -ps}{(p -1)p}}\right) >0
\end{equation*}
and $ \lambda < \lambda_{****},$  since $ C^{\prime\prime}_{5} >0$ where
{\tiny
\begin{equation*}
\begin{aligned}
&\lambda_{****} =\frac{C^{\prime\prime\prime}}{ C^{\prime\prime}_{5}} \epsilon^{\frac{n -ps}{(p -1)p}} +\frac{1}{ C^{\prime\prime}_{5}}\epsilon^{ -\frac{n -ps}{p } +m\frac{\alpha -1}{p\theta +\alpha -1}}\left(\frac{p_s^* +\alpha -1}{p_s^*}\Vert u\Vert_{\infty}\vert \Omega\vert^{\frac{p_s^* +\alpha -1}{p_s^*}}S_{p}^{\frac{\alpha -1}{p}}\right)^{\frac{p\theta}{p\theta +\alpha -1}}\\
&\left( \frac{p_s^* -p\theta}{p_s^*p\theta}\right)^{\frac{\alpha -1}{p\theta +\alpha -1}}\left( \frac{\alpha}{1 -\alpha}\right)
\end{aligned}
\end{equation*}
}
Put
\begin{equation*}
\lambda^{***} =\min\left\lbrace \lambda^{**}, \lambda_{****} \right\rbrace
\end{equation*}
this finishes the proof
\end{proof}
now we prove the existence of second solutions in $ \mathcal{N}_{\lambda}^{-}.$ Let $ \lambda < \lambda^{***}$ by lemma \ref{l9} $ u_* \in U_1$ and $ u_* +l_0u_{\epsilon,\sigma} \in U_2,$ so there exist continous path connecting $ U_1$ and $ U_2$ in particular there exist $ t \in (0, 1)$ such that 
$ u_* +tl_0u_{\epsilon,\sigma} \in \mathcal{N}_{\lambda}^{-}$ which means that
$ m^{-} \leq J_{\lambda}(u_0 +tl_0u_{\epsilon, \sigma})$ on the other hand by lemma \ref{l10} and \ref{ek2} we have $ m^{-} <c_{\lambda}$ for 
$ \lambda \in (0, \lambda^{***}).$ Moreover by lemma \ref{l3} $ \mathcal{N}_{\lambda}^{-}$ is closed set in $ X_0$ so by corollary \ref{c1} we find the infinimum of  $ J_{\lambda}$ on $\mathcal{N}_{\lambda}^{-}$ using \ref{ek1} and \ref{ek2} for $  (u_n)_n \in \mathcal{N}_{\lambda}^{-}$ and $ u_n \rightharpoonup u_*$ for some $ u_* \in X_0,$ so we get that $ \lim_{n \to \infty} J_{\lambda}(u_n) =m^- $ therefore using proposition \ref{pr1} with $ m^- <c_{\lambda}$ yields 
 up to subsequence still denoted $ u_n$ that $ u_n \longrightarrow u_*$ strongly in $ X_0$ in addition $ \mathcal{N}_{\lambda}^{-}$ closed which means that 
 $ u_* \in \mathcal{N}_{\lambda}^{-}$ and $ J_{\lambda}(u_*) =m^{-},$ by the same argument above $ u_*$ is a weak solution to \ref{p} namely verifies \ref{eee60} yields that $ c(x)(u_*^+)^{-\alpha}\Phi \in L^{1}(\Omega)$ for any $ \Phi \in X_0,$ this and using lemma \ref{l2} we get that $ u_*$ is a nontrivial weak solution to problem \ref{p}, replacing $ \phi $ by $ (u_*)^-$ in \ref{weak} we get that $ \Vert (u_*)^-\Vert =0,$ that is  $ u_* \geq 0.$ Applying the maximum priciple we get that $ u_*$ is a positive solution to problem \ref{p0}. Moreover $ u_* $ and $ u^*$ are distict due to fact that 
 $ \mathcal{N}_{\lambda}^{-} \cap \mathcal{N}_{\lambda}^{+} =\emptyset.$ This completes the proof of theorem \ref{th2}.
\newpage
\addcontentsline{toc}{chapter}{Bibliography}
\bibliographystyle{acm}
\bibliography{NoteEllipticProblem}

\begin{thebibliography}{10}

\bibitem{arcoya2014multiplicity}
{\sc Arcoya, D., and Moreno-M{\'e}rida, L.}
\newblock Multiplicity of solutions for a dirichlet problem with a strongly
  singular nonlinearity.
\newblock {\em Nonlinear Analysis: Theory, Methods \& Applications 95\/}
  (2014), 281--291.

\bibitem{arora2020polyharmonic}
{\sc Arora, R., Giacomoni, J., Mukherjee, T., and Sreenadh, K.}
\newblock Polyharmonic kirchhoff problems involving exponential non-linearity
  of choquard type with singular weights.
\newblock {\em Nonlinear Analysis 196\/} (2020), 111779.

\bibitem{assunccao2017existence}
{\sc Assun{\c{c}}{\~a}o, R.~B., Miyagaki, O.~H., and Rodrigues, B.~M.}
\newblock Existence and multiplicity of solutions for a supercritical elliptic
  problem in unbounded cylinders.
\newblock {\em Boundary Value Problems 2017}, 1 (2017), 1--11.

\bibitem{autuori2015stationary}
{\sc Autuori, G., Fiscella, A., and Pucci, P.}
\newblock Stationary kirchhoff problems involving a fractional elliptic
  operator and a critical nonlinearity.
\newblock {\em Nonlinear Analysis 125\/} (2015), 699--714.

\bibitem{boccardo2012dirichlet}
{\sc Boccardo, L.}
\newblock A dirichlet problem with singular and supercritical nonlinearities.
\newblock {\em Nonlinear Analysis: Theory, Methods \& Applications 75}, 12
  (2012), 4436--4440.

\bibitem{brezis2011functional}
{\sc Brezis, H., and Br{\'e}zis, H.}
\newblock {\em Functional analysis, Sobolev spaces and partial differential
  equations}, vol.~2.
\newblock Springer, 2011.

\bibitem{brezis1983relation}
{\sc Br{\'e}zis, H., and Lieb, E.}
\newblock A relation between pointwise convergence of functions and convergence
  of functionals.
\newblock {\em Proceedings of the American Mathematical Society 88}, 3 (1983),
  486--490.

\bibitem{choi2019existence}
{\sc Choi, Q., Jung, T., et~al.}
\newblock Existence of solution for p-laplacian boundary value problems with
  two singular and subcritical nonlinearities.
\newblock {\em Boundary Value Problems 2019}, 1 (2019), 1--16.

\bibitem{di2012hitchhikers}
{\sc Di~Nezza, E., Palatucci, G., and Valdinoci, E.}
\newblock Hitchhikers guide to the fractional sobolev spaces.
\newblock {\em Bulletin des sciences math{\'e}matiques 136}, 5 (2012),
  521--573.

\bibitem{brezislieb}
{\sc Emelyanov, E., and Marabeh, M.}
\newblock On the brezis--lieb lemma and its extensions.
\newblock In {\em Operator Theory and Differential Equations}. Springer, 2021,
  pp.~25--35.

\bibitem{figueiredo2012multiplicity}
{\sc Figueiredo, G.~M., and Junior, J. R.~S.}
\newblock Multiplicity of solutions for a kirchhoff equation with subcritical
  or critical growth.
\newblock {\em Differential and Integral Equations 25}, 9/10 (2012), 853--868.

\bibitem{fiscella2019fractional}
{\sc Fiscella, A.}
\newblock A fractional kirchhoff problem involving a singular term and a
  critical nonlinearity.
\newblock {\em Advances in Nonlinear Analysis 8}, 1 (2019), 645--660.

\bibitem{fiscella2019nehari}
{\sc Fiscella, A., and Mishra, P.~K.}
\newblock The nehari manifold for fractional kirchhoff problems involving
  singular and critical terms.
\newblock {\em Nonlinear Analysis 186\/} (2019), 6--32.

\bibitem{ghanmi2016multiplicity}
{\sc Ghanmi, A., and Saoudi, K.}
\newblock A multiplicity results for a singular problem involving the
  fractional p-laplacian operator.
\newblock {\em Complex variables and elliptic equations 61}, 9 (2016),
  1199--1216.

\bibitem{ghanmi2016nehari}
{\sc Ghanmi, A., and Saoudi, K.}
\newblock The nehari manifold for a singular elliptic equation involving the
  fractional laplace operator.
\newblock {\em RN 1\/} (2016), 2.

\bibitem{giacomoni2017positive}
{\sc Giacomoni, J., Mukherjee, T., and Sreenadh, K.}
\newblock Positive solutions of fractional elliptic equation with critical and
  singular nonlinearity.
\newblock {\em Advances in Nonlinear Analysis 6}, 3 (2017), 327--354.

\bibitem{giacomoni2009multiplicity}
{\sc Giacomoni, J., and Saoudi, K.}
\newblock Multiplicity of positive solutions for a singular and critical
  problem.
\newblock {\em Nonlinear Analysis: Theory, Methods \& Applications 71}, 9
  (2009), 4060--4077.

\bibitem{godoy2017multiplicity}
{\sc Godoy, T., and Guerin, A.}
\newblock Multiplicity of positive weak solutions to subcritical singular
  elliptic dirichlet problems.
\newblock {\em Electronic Journal of Qualitative Theory of Differential
  Equations 2017}, 100 (2017), 1--30.

\bibitem{hirano2004existence}
{\sc Hirano, N., Saccon, C., and Shioji, N.}
\newblock Existence of multiple positive solutions for singular elliptic
  problems with concave and convex nonlinearities.
\newblock {\em Advances in Differential Equations 9}, 1-2 (2004), 197--220.

\bibitem{lei2015multiple}
{\sc Lei, C.-Y., Liao, J.-F., and Tang, C.-L.}
\newblock Multiple positive solutions for kirchhoff type of problems with
  singularity and critical exponents.
\newblock {\em Journal of Mathematical Analysis and Applications 421}, 1
  (2015), 521--538.

\bibitem{mosconi2016brezis}
{\sc Mosconi, S., Perera, K., Squassina, M., and Yang, Y.}
\newblock A brezis nirenberg result for the fractional p-laplacian.
\newblock {\em Calc. Var. Partial Differ. Equ 55}, 55 (2016), 105.

\bibitem{mukherjee2016fractional}
{\sc Mukherjee, T., and Sreenadh, K.}
\newblock Fractional elliptic equations with critical growth and singular
  nonlinearities.
\newblock {\em Electron. J. Differential Equations 54\/} (2016), 1--23.

\bibitem{saoudi2017multiplicity}
{\sc Saoudi, K., Agarwal, P., and Mursaleen, M.}
\newblock A multiplicity result for a singular problem with subcritical
  nonlinearities.
\newblock {\em Journal of Nonlinear Functional Analysis\/} (2017), 1--18.

\bibitem{servadei2012mountain}
{\sc Servadei, R., and Valdinoci, E.}
\newblock Mountain pass solutions for non-local elliptic operators.
\newblock {\em Journal of Mathematical Analysis and Applications 389}, 2
  (2012), 887--898.

\bibitem{servadei2015brezis}
{\sc Servadei, R., and Valdinoci, E.}
\newblock The brezis-nirenberg result for the fractional laplacian.
\newblock {\em Transactions of the American Mathematical Society 367}, 1
  (2015), 67--102.

\bibitem{sun2011exact}
{\sc Sun, Y., and Wu, S.}
\newblock An exact estimate result for a class of singular equations with
  critical exponents.
\newblock {\em Journal of Functional Analysis 260}, 5 (2011), 1257--1284.

\bibitem{tarantello1992nonhomogeneous}
{\sc Tarantello, G.}
\newblock On nonhomogeneous elliptic equations involving critical sobolev
  exponent.
\newblock In {\em Annales de l'Institut Henri Poincar{\'e} C, Analyse non
  lin{\'e}aire\/} (1992), vol.~9, Elsevier, pp.~281--304.

\bibitem{xiang2016existence}
{\sc Xiang, M., Zhang, B., and R{\u{a}}dulescu, V.~D.}
\newblock Existence of solutions for perturbed fractional p-laplacian
  equations.
\newblock {\em Journal of Differential Equations 260}, 2 (2016), 1392--1413.

\end{thebibliography}
\end{document}